\documentclass{amsart}
\usepackage{amssymb, amsthm}
\usepackage{graphicx}
\usepackage{amsmath}
\usepackage{xypic}
\usepackage{tikz}
\usetikzlibrary{matrix}
\usepackage[left=3.4cm,right=3.4cm,top=2cm,bottom=2cm]{geometry}
\usepackage{hyperref}

\newtheorem{thm}{Theorem}[section]

\newtheorem{cor}[thm]{Corollary}
\newtheorem{lemma}[thm]{Lemma}
\newtheorem{prop}[thm]{Proposition}

\theoremstyle{definition}
\newtheorem{defn}[thm]{Definition}
\newtheorem*{claim}{Claim}
\newtheorem{ex}[thm]{Example}

\theoremstyle{remark}
\newtheorem{rem}[thm]{Remark}

\newcommand{\ZZ}{{\mathbb  Z}}
\newcommand{\RR}{{\mathbb  R}}

\newcommand{\CC}{{\mathbb  C}}

\renewcommand{\cL}{{\mathcal L}}

\newcommand{\cF}{{\mathcal F}}

\newcommand{\F}{{\mathcal F}}

\newcommand{\mfg}{\mathfrak{g}}
\newcommand{\mfh}{\mathfrak{h}}

\newcommand{\mfa}{\mathfrak{a}}

\newcommand{\mft}{\mathfrak{t}}

\newcommand{\Ad}{\operatorname{Ad}}

\newcommand{\coker}{\operatorname{coker}}
\newcommand{\rk}{\operatorname{rank}}
\newcommand{\bas}{\operatorname{bas}}

\newcommand{\im}{\operatorname{im}}

\newcommand{\id}{\operatorname{id}}
\newcommand{\odd}{\operatorname{odd}}
\newcommand{\even}{\operatorname{even}}

\newcommand{\rank}{\operatorname{rank}}

\newcommand{\U}{\mathrm{U}}
\newcommand{\SU}{\mathrm{SU}}

\newcommand{\GL}{\mathrm{GL}}

\newcommand		{\quotientmed}[2]	{{\raisebox{.2em}{$#1$}}\  \!\!\big/\!\!\ 
	{\raisebox{-.2em}{$#2$}}}

\title[Equivariant de Rham cohomology]{Equivariant de Rham cohomology:\\ Theory and applications}
\author{Oliver Goertsches}
\email{goertsch@mathematik.uni-marburg.de}
\author{Leopold Zoller}
\email{zoller@mathematik.uni-marburg.de}
\address{Philipps Universit\"at Marburg\\ Fachbereich Mathematik und Informatik\\ Hans-Meerwein-Stra\ss e\\ 35032 Marburg\\ Germany}

\begin{document}

\numberwithin{equation}{section}

\begin{abstract}
This is a survey on the equivariant cohomology of Lie group actions on manifolds, from the point of view of de Rham theory. Emphasis is put on the notion of equivariant formality, as well as on applications to ordinary cohomology and to fixed points.
\end{abstract}

\maketitle

\tableofcontents

\section{Introduction}

Equivariant cohomology is a topological invariant, not of spaces, but of group actions. It encodes in a subtle way information on the topology of the space, the isotropy groups of the action, and the orbit stratification, in particular on the fixed points of the action. In was introduced by Borel \cite{Borel} and H.\ Cartan \cite{Cartan1}, \cite{Cartan2} in the 1950s and has found numerous applications wherever symmetries of geometric objects play a role.  These purpose of these notes is twofold: they try to give a gentle introduction to this beautiful theory from the point of view of de Rham theory, and to survey both classical and more recent applications.

In the first few sections we introduce three different types of cohomology one can associate to a Lie group action on a manifold: cohomology of invariant forms, basic cohomology, and our main player, equivariant cohomology. After comparing them to each other and to ordinary (de Rham) cohomology we prove some basic results on equivariant cohomology like the homotopy axiom and the Mayer-Vietoris sequence.

We explain how equivariant cohomology can be used to gain information on both the ordinary cohomology of the manifold $M$ acted on, as well as on the fixed point set of the action. The main tool to relate equivariant cohomology to the fixed point set is the Borel localization theorem, which is the topic of Section \ref{sec:borellocalization}. We explain how one uses it to show the equalities of the Euler characteristics of $M$ and the fixed point set $M^T$, as well as the inequality of total Betti numbers $\dim H^*(M^T)\leq \dim H^*(M)$, in Section \ref{sec:fixedpoints}. 

Starting with Section \ref{sec:equivariantformality} we make use of the spectral sequence of the Cartan model, as there we introduce another main topic of this survey, the notion of equivariant formality.  All necessary knowledge on spectral sequences is contained in the appendix; in particular, there one can find details on the relation between the equivariant cohomology and the final page of the spectral sequence that are usually glossed over in the literature.  Equivariant formality of an action is the condition that the spectral sequence of the Cartan model degenerates at $E_1$. In Theorem \ref{thm:bigthmequivformal} we prove some equivalent formulations of this condition, one of which enables one to compute ordinary from equivariant cohomology. We apply this to obtain information on the cohomology of homogeneous spaces in Section \ref{sec:cohomhomspaces}, and of GKM manifolds in Section \ref{sec:HMHTM}.

In the last sections we give a short overview on some recent developments. The choice of material is rather biased and not meant to be exhaustive. We will explain some results surrounding the notions of Cohen-Macaulay actions and equivariant basic cohomology.

Throughout the paper we try to present the material in an easily accessible way, sometimes sacrificing greater generality for simplicity of the arguments. We do not give proofs for every result, but do so whenever we were not able to find a good reference in the literature; sometimes we provide a different proof. We will assume that the reader is familiar with the theory of actions of compact Lie groups on differentiable manifolds.

In preparation of this paper a wealth of literature was helpful, such as the monographs \cite{AlldayPuppe, BeGeVe, GuilleminSternberg, Hsiang}, as well as \cite[Appendix C]{GGK} and \cite{Bott}. \\

\noindent {\it Acknowledgements.} 
Parts of this paper stem from the first named author's lectures at the University of Hamburg in 2012, and at the Philipps University of Marburg in 2018. We would like to thank the participants of these courses for their interest in the topic and their valuable comments. We are grateful to Mich\`ele Vergne for several remarks on a previous version of this paper.  We are especially indebted to Jeffrey Carlson for several enlightening discussions, as well as for a very thorough reading of a previous version and numerous suggestions that improved the presentation of this paper. The second named author is supported by the German Academic Scholarship foundation.

\section{Invariant and basic differential forms}

Let $G$ be a Lie group acting on a differentiable manifold $M$, with Lie algebra $\mfg$. We denote, for $X\in \mfg$, the induced fundamental vector field by
\[
\overline{X}_p:= \left.\frac{d}{dt}\right|_{t=0} \exp(tX)\cdot p.
\] 
\begin{defn} A differential form $\omega\in \Omega(M)$ is called \emph{$G$-invariant} if $g^*\omega=\omega$ for all $g\in G$. The space of $G$-invariant differential forms is denoted $\Omega(M)^G$. 
\end{defn}
The space $\Omega(M)^G$ is clearly invariant under the differential $d\colon\Omega(M)\to \Omega(M)$, i.e., $(\Omega(M)^G,d)$ is a subcomplex of $(\Omega(M),d)$ and we can consider its cohomology. However, if $G$ is connected and compact, this cohomology does not contain more information than the usual de Rham cohomology because of the following theorem due to \'E.\ Cartan \cite{ECartan}:
\begin{thm}\label{thm:cohomofinvariantforms}
If $G$ is a compact and connected Lie group acting on a differentiable manifold $M$, then the inclusion map $\Omega(M)^G\to \Omega(M)$ induces an isomorphism $H^*(\Omega(M)^G)\to H^*(M)$ in cohomology.
\end{thm} 
The proof of this result can be found e.g.~in \cite[\S 9]{Onishchik}. One shows that the averaging operator $\mu\colon\Omega(M)\to \Omega(M)$ given by
\[
\mu(\omega)(v_1,\ldots,v_n):= \int_G (g^*\omega) (v_1,\ldots,v_n);
\]
is chain homotopic to the identity. Of course, if $G$ is not connected, then this inclusion does not induce an isomorphism, see Examples \ref{ex:cohomfinitequotient} and \ref{ex:cohomRPn} below.

A different type of topological information is encoded in the complex of $G$-basic differential forms.
\begin{defn}\label{defn:basicforms} Given an action of a Lie group $G$ on a smooth manifold $M$, a differential form $\omega\in \Omega(M)$ is called \emph{($G$-)horizontal} if $i_{\overline{X}}\omega=0$ for all $X\in \mfg$. It is called \emph{$G$-basic} if it is both $G$-invariant and horizontal. The space of such differential forms is denoted $\Omega_{\bas G}(M)$.
\end{defn}

Just like the $G$-invariant differential forms, also the basic differential forms comprise a subcomplex of the de Rham complex. In fact, for $\omega\in \Omega_{\bas G}(M)$, the form $d\omega$ is again (invariant and) horizontal because by the Cartan formula
$i_{\overline{X}} d\omega = \cL_{\overline{X}}\omega - di_{\overline{X}}\omega = 0$. Here, $\cL$ denotes the Lie derivative.

\begin{defn}
We define the \emph{basic cohomology} of the $G$-action to be
\[
H^*_{\bas G}(M):=H^*(\Omega_{\bas G}(M),d).
\]
\end{defn}
Recall that if the $G$-action on $M$ is free, then the orbit space $M/G$ is a smooth manifold, and the projection $\pi\colon M\to M/G$ is smooth. In general, for an arbitrary action of a compact Lie group, $M/G$ is just a topological Hausdorff space.
\begin{prop}\label{prop:basiccohomfreeaction} Consider a free action of a (not necessarily connected) compact Lie group $G$ on a smooth manifold $M$, and consider the projection $\pi\colon M\to M/G$. Then $\pi^*$ defines an isomorphism of complexes $\pi^*\colon \Omega(M/G)\to \Omega_{\bas G}(M)$. In particular, 
\[
H^*_{\bas G}(M)\cong H^*(M/G).
\]
\end{prop}
\begin{proof} If $\omega\in \Omega(M/G)$, then $\pi^*\omega$ is $G$-invariant because for any $g\in G$ we have
\[
g^*\pi^*\omega = (\pi\circ g)^*\omega = \pi^*\omega.
\]
At each $p\in M$, we have $\ker d\pi_p = T_pG\cdot p$. Thus, $\pi^*\omega$ is horizontal as well. 

If conversely $\eta$ is a $G$-basic $k$-form on $M$, then we can define a $k$-form $\omega$ on $M/G$ as follows: if $v_1,\ldots,v_k$ are tangent vectors at $Gp\in M/G$, then let $w_1,\ldots,w_k$ be tangent vectors at $p\in M$ such that $d\pi_p(w_i)=v_i$, and define
\[
\omega(v_1,\ldots,v_k)=\eta(w_1,\ldots,w_n)
\]
This is independent of both the choice of $p$ and the $w_i$ because $\eta$ is $G$-invariant and horizontal. Clearly, we have $\pi^*\omega = \eta$.
\end{proof}
\begin{ex}\label{ex:cohomfinitequotient}
Consider a finite group $G$ acting freely on a smooth manifold $M$. Then being $G$-basic, for a differential form $\omega$ on $M$, is the same as being $G$-invariant. So in this case $\pi^*\colon \Omega(M/G) \to \Omega(M)^G$ is an isomorphism of complexes, so that $H^*(M/G) = H^*(\Omega(M)^G)$.

On the other hand, we have a well-defined action of $G$ on cohomology: for $g\in G$ and $[\omega]\in H^*(M)$, we put $g^*[\omega] := [g^*\omega]$. Then the inclusion $\Omega(M)^G\to \Omega(M)$ induces an injective homomorphism $H^*(\Omega(M)^G) \to H^*(M)$ whose image lies in the $G$-invariant cohomology:
\[
i_*\colon H^*(\Omega(M)^G)\longrightarrow H^*(M)^G.
\]
We claim that this map is indeed surjective: let $[\omega]\in H^*(M)^G$, i.e., that for all $g\in G$ there exists $\eta_g\in \Omega(M)$ such that $g^*\omega = \omega + d\eta_g$. But then the average $\frac1{|G|} \sum_{g\in G} g^*\omega$ is a $G$-invariant form whose cohomology class is sent by $i_*$ to $[\omega]$. In total, we obtain isomorphisms
\[
H^*(M/G) \longrightarrow H^*(\Omega(M)^G) \longrightarrow H^*(M)^G.
\]
\end{ex}
\begin{ex}\label{ex:cohomRPn}
Let us give a concrete example: we consider the free $\ZZ_2$-action on the $n$-dimensional sphere $S^n$ given by sending a point to its antipodal map, with orbit space the real projective space $\RR P^n$. To understand the cohomology of $\RR P^n$ we therefore only have to understand the effect of the map $f(x)=-x$ on a volume form of $S^n$. A volume form on $S^n$ is given by 
\begin{align*}
\omega_{(x_1,\ldots,x_{n+1})}&=i_{(x_1,\ldots,x_{n+1})} (dx_1\wedge \cdots \wedge dx_{n+1}),
\end{align*}
and as the radial vector field is invariant under $f$, it follows that $f^*\omega = \omega$ for odd $n$, and $f^*\omega = -\omega$ for even $n$. The action of $\ZZ_2$ on $H^n(S^n)=\RR\cdot [\omega]$ is thus trivial for odd $n$, and given by multiplication with $-1$ for even $n$, whence
\[
H^n(\RR P^n) = H^n(S^n)^{\ZZ_2} = \begin{cases} 0 & n \text{ even,} \\ \RR & n \text{ odd}. \end{cases}
\]
\end{ex}
\begin{rem}\label{rem:basicsingularcohom}
We even have $H^*_{\bas G}(M)=H^*(M/G)$ for any action of a compact Lie group $G$, where the right hand side is understood as the singular cohomology of $M/G$. As singular cohomology is not the focus of these notes, we only refer to \cite[Theorem 30.36]{Michor} for the proof. This tells us that $H^*_{\bas G}(M)$ is, in many cases, not a very powerful invariant for group actions. For instance, there exist many nontrivial group actions for which the orbit space $M/G$ is contractible, so that $H^*_{\bas G}(M)=\RR$, e.g., the standard action of $S^1$ on $S^2$ by rotation. For free actions, however, the orbit space is again a manifold, so basic cohomology of free actions is an invariant as powerful as de Rham cohomology for manifolds.
\end{rem}

\section{The coadjoint representation}\label{sec:Coadjoint}
Any Lie group $G$ acts on its Lie algebra by the adjoint representation. This is defined as follows: for any $g\in G$ conjugation with $g$ is denoted
\[
c_g\colon G\longrightarrow G;\,  h\longmapsto ghg^{-1}. 
\]
Differentiating this at $e$, we obtain a map $\Ad_g\colon \mfg\cong T_eG\to T_eG\cong \mfg$ given by $\Ad_g:=(dc_g)_e$. In this way we obtain a homomorphism
\[
\Ad\colon G\longrightarrow \GL(\mfg)
\]
which we call the \emph{adjoint representation} of $G$.

Dualizing this representation, we obtain the \emph{coadjoint representation} of $G$ on the dual vector space $\mfg^*$ (which consists of linear forms $\xi\colon \mfg\to \RR$):
\[
(\Ad_g^*\xi) (X):=\xi(\Ad_{g^{-1}}(X))
\]
We denote by $S(\mfg^*)$ the symmetric algebra on $\mfg^*$, which we consider as the algebra of polynomials on $\mfg$. The coadjoint representation naturally extends to $S(\mfg^*)$ via  $(\Ad_g^*f)(X):=f(\Ad_{g^{-1}}X)$.
Of particular importance will be the subspace of $G$-invariant polynomials $S(\mfg^*)^G$, i.e., those polynomials that are constant along adjoint orbits in $\mfg$.

For compact and connected $G$, the ring of invariant polynomials is again a polynomial ring: Chevalley's restriction theorem, see \cite[Chapitre VIII, \S 8.3, Th\'eor\`eme 1]{Bourbaki}, \cite[Theorem 4.9.2]{Varadarajan}  or \cite{DufloVergne} (it was mentioned by Chevalley without proof in \cite[Section IV]{ChevalleyProceedings}), states that the restriction map
\[
S(\mfg^*)^G\longrightarrow S(\mft^*)^{W(G)},
\]
where $T\subset G$ is a maximal torus and $W(G)$ the corresponding Weyl group, is an isomorphism. (See \cite[Proposition 30]{DufloVergne} for an explicit description of the inverse map.) Here, we define the Weyl group as the finite group $N_G(T)/T$, where $N_G(T) = \{g\in G\mid gTg^{-1}=G\}$ is the normalizer of $T$ in $G$.  As the Weyl group acts on $\mft^*$ as a reflection group (it coincides with the algebraically defined Weyl group of the root system of $\mfg^\CC$, see \cite[Theorem IV.4.54]{Knapp}), the Chevalley-Shephard-Todd theorem \cite[Section 18-1]{Kane} states that the ring of invariants $S(\mft^*)^{W(G)}$ is a polynomial $\RR$-algebra. 

\begin{ex}\label{ex:invpolyUn}
Consider $G=\U(n)$, with maximal torus $T$ given by diagonal matrices, and corresponding Weyl group $S_n$, acting by permutations on the diagonal entries of $\mft$. Then $S(\mfg^*)^G \cong S(\mft^*)^{W(\U(n))}$ is the algebra of symmetric polynomials in $n$ variables, which is  the polynomial algebra $\RR[\sigma_1,\ldots,\sigma_n]$, generated by the elementary symmetric polynomials $\sigma_i$ of degree $i$. A direct proof of Chevalley's restriction theorem for the case $G=\U(n)$ can be found in \cite[Example C.13]{GGK}.
\end{ex}

\begin{ex}
For a disconnected compact Lie group $G$, the $G$-invariant polynomials do not necessarily form a polynomial ring. Consider, for example, the semidirect product $G= T^2\rtimes_\varphi \ZZ_2$, where $\varphi(1)$ acts as the inverse map on $T^2$. Then $S(\mfg^*)^G = \RR[x,y]^{\ZZ_2}$, where $\ZZ_2$ acts on $x$ and $y$ by $\pm 1$, which is the algebra of polynomials in $x$ and $y$ of even degree. This is not a polynomial ring, because any generating set necessarily contains scalar multiples of $x^2,y^2$ and $xy$, and we have the relation $(xy)^2 = x^2y^2$.
\end{ex}

\section{The Cartan model}\label{sec:CartanModel}
In this section we introduce H.\ Cartan's definition of equivariant cohomology \cite{Cartan1}, \cite{Cartan2}. Let $G$ be a compact Lie group acting on a differentiable manifold $M$. We define the space of \emph{equivariant differential forms on $M$} as 
\[
C_G(M):=(S(\mfg^*)\otimes \Omega(M))^G.
\]
Here, the superscript denotes taking the subspace of $G$-invariant objects, where $S(\mfg^*)\otimes \Omega(M)$ is endowed with the tensor product representation: $G$ acts on $S(\mfg^*)$ by the coadjoint representation described in the previous subsection and on $\Omega(M)$ by pull-back, i.e., the following representation: 
\[
g\cdot \omega:=(g^{-1})^*\omega.
\]
An equivariant differential form $\omega\in S(\mfg^*)\otimes \Omega(M)$ can be written as a finite sum
\[
\omega = \sum_i f_i\otimes \eta_i,
\]
for $f_i\in S(\mfg^*)$ and $\eta_i\in \Omega(M)$. By abuse of notation, we will also denote the associated polynomial map $\mfg\to \Omega(M);\, X\mapsto \sum_i f_i(X)\cdot \eta_i$ by $\omega$. Almost by definition, the $G$-invariance of the element $\omega\in S(\mfg^*)\otimes \Omega(M)$ translates to the equivariance of the polynomial map $\omega\colon \mfg\to \Omega(M)$, i.e., to the condition
\begin{equation}\label{eq:Ad-invariantpolynomial}
\omega(\Ad_g(X))=g\cdot (\omega(X)) = (g^{-1})^*(\omega(X))
\end{equation}
for all $g\in G$ and $X\in \mfg$. We think of $C_G(M)$ as the space of $G$-equivariant polynomial maps $\mfg\to \Omega(M)$.

\begin{rem}
If $G=T$ is a torus, then the (co)adjoint action of $T$ is trivial, so $C_T(M) = S(\mft^*)\otimes \Omega(M)^T$. A $T$-equivariant differential form is nothing but a polynomial $\omega\colon \mft\to \Omega(M)^T$.
\end{rem}

Sometimes it is convenient to write equivariant differential forms in a basis: given a basis $\{X_i\}$ of the Lie algebra $\mfg$, with dual basis $\{u_i\}$ of $\mfg^*$, we can write an equivariant differential form $\omega\in C_G(M)$ as a finite sum
\begin{equation}\label{eq:localdescriptionofeqdiffform}
\omega = \omega_\emptyset + \sum_i \omega_i u_i + \sum_{i\leq j} \omega_{ij} u_iu_j + \cdots = \sum_I \omega_I u_I,
\end{equation}
where $I$ runs over a finite set of multiindices.

There is a natural $S(\mfg^*)^G$-algebra structure on $C_G(M)$: first of all note that $C_G(M)$ is a ring with respect to the multiplication
\[
(\omega\wedge \eta)(X):=\omega(X)\wedge \eta(X),
\]
where $\omega$ and $\eta$ are considered as polynomials $\mfg\to \Omega(M)$. In other words, we give $C_G(M)$ the ring structure from the tensor product of the rings $S(\mfg^*)$ and $\Omega(M)$. The $S(\mfg^*)^G$-algebra structure is defined by the ring homomorphism
\begin{equation}\label{eq:algebrastroneqc}
i\colon S(\mfg^*)^G\to (S(\mfg^*)\otimes \Omega(M))^G = C_G(M);\, f\mapsto f\otimes 1.
\end{equation}
As a polynomial $\mfg\to \Omega(M)$, the equivariant differential form $f\otimes 1$ is $(f\otimes 1)(X)=f(X)$, where the real number $f(X)$ is regarded as a constant function on $M$.

\begin{defn}\label{def:equivariantdifferential}
We define the \emph{equivariant differential} $d_G$ on $S(\mfg^*)\otimes\Omega(M)$ by
\[
d_G(\omega)(X)=d(\omega(X)) - i_{\overline{X}} \omega(X).
\]
\end{defn}
\begin{rem}
There are various sign conventions in the literature. Some authors use $+$ instead of $-$ in this definition; also, some authors use a sign in the definition of the fundamental vector field $\overline{X}$, to make the assignment $X\mapsto \overline{X}$ a Lie algebra homomorphism.
\end{rem}

One directly verifies that $d_G$ maps $C_G(M)$ to itself.  It is useful to write the equivariant differential $d_G\omega$ in case $\omega$ is given explicitly as in \eqref{eq:localdescriptionofeqdiffform}:
\begin{lemma}\label{lem:d_Gexplicit} If $\omega= \sum_I \omega_Iu_I\in S(\mfg^*)\otimes \Omega(M)$, then
\begin{equation}\label{eq:d_Gexplicit}
d_G\omega = \sum_I (d\omega_I -\sum_i i_{\overline{X_i}}\omega_I u_i)u_I.
\end{equation}
\end{lemma}
\begin{proof}
We only need to observe that for $X\in \mfg$, we have $X=\sum_i u_i(X) X_i$, so that $i_{\overline{X}} = \sum_i u_i(X) i_{\overline{X_i}}$.
\end{proof}

Let us introduce a grading on $C_G(M)$. For any integer $n\geq 0$ we define the space of equivariant differential forms of degree $n$ as
\[
C_G^n(M):=\bigoplus_{2k+l=n} (S^k(\mfg^*)\otimes \Omega^l(M))^G.
\]
An element of $C_G^n(M)$ will be called an \emph{equivariant differential form of degree $n$}. Note that the ring structure of $C_G(M)$ is graded in the sense that the product of elements in degree $n$ and $m$ is of degree $n+m$.
\begin{rem} If $\omega=\sum_I \omega_I u_I$ is an equivariant differential form as in \eqref{eq:localdescriptionofeqdiffform}, then it is of degree $n$ if and only if for every $I=(i_1,\ldots,i_r)$ the differential form $\omega_I$ is of degree $n-2(i_1+\cdots+i_r)$.
\end{rem}

In the following proposition we collect a few properties of the equivariant differential. We omit the straightforward proofs. The first item is the reason for our choice of grading on $C_G(M)$.
\begin{prop}
\begin{enumerate}
\item $d_G$ maps $C_G^n(M)$ to $C_G^{n+1}(M)$.
\item For $\omega\in C_G^n(M)$ and $\eta\in C_G^m(M)$ we have
\[
d_G(\omega\wedge\eta)=(d_G\omega)\wedge \eta + (-1)^n \omega\wedge (d_G\eta).
\]
\item $d_G^2=0$.
\end{enumerate}
\end{prop}

If $d_G\omega=0$, then we say that $\omega$ is \emph{equivariantly closed}, and a form of the type $d_G\eta$ is \emph{equivariantly exact}.

\begin{defn} The \emph{equivariant cohomology} of the $G$-action on $M$ is defined as 
$H^*_G(M):=H^*(C_G^*(M),d_G)$.
\end{defn}

The ring structure of $C_G(M)$ passes over to $H^*_G(M)$, and the ring homomorphism $i$ in \eqref{eq:algebrastroneqc} induces a well-defined homomorphism of graded rings $i\colon S(\mfg^*)^G\to H^*_G(M)$. Thus, via $i$, the ring $H^*_G(M)$ becomes naturally a graded $S(\mfg^*)^G$-algebra, in the sense that the ring homomorphism $i$ respects the degree. In what follows, it will be extremely important to distinguish between this $S(\mfg^*)^G$-algebra structure on $H^*_G(M)$ and the induced structure as an $S(\mfg^*)^G$-module.

\begin{rem}\label{rem:Borelmodel}
There are other ways to introduce equivariant cohomology, most prominently the so-called Borel model, introduced first in \cite{Borel}, which we now briefly explain.   As was mentioned above in Remark \ref{rem:basicsingularcohom}, we consider the cohomology of the orbit space a reasonable invariant for free actions. In case of an arbitrary action on a topological space $X$, one now replaces the space $X$ acted on by a homotopy equivalent space with a free $G$-action, namely by
\[
EG\times X,
\]
where $EG$ is a contractible space on which $G$ acts freely. Then, one defines the equivariant cohomology (with coefficients $R$) as the cohomology of the orbit space of the diagonal action:
\[
H^*_G(X;R) := H^*(EG\times_G G;R).
\]
It admits the structure of a $H^*(BG;R)$-algebra, via the natural projection $EG\times_G X\to EG/G=:BG$. The equivariant de Rham theorem \cite{Cartan1}, \cite{Cartan2}, see also \cite[Section 2.5]{GuilleminSternberg}, states that for manifolds and real coefficients, this Borel cohomology is isomorphic to the equivariant cohomology defined above. A further important model for equivariant cohomology is the Weil model. See \cite{Meinrenken} for a short overview of these models.
\end{rem}

\begin{ex}\label{ex:trivialaction} Let us consider an easy, yet very important example: that of a trivial $G$-action on a manifold $M$. In this case, any differential form on $M$ is automatically $G$-invariant, so we have
\[
C_G(M)=S(\mfg^*)^G\otimes \Omega(M).
\]
All induced vector fields $\overline{X}$ are trivial, so the equivariant differential $d_G$ is nothing but the ordinary differential: $(d_G\omega)(X)=d(\omega(X))$. This means that the complex $(C_G(M),d_G)$ is obtained from the ordinary de Rham complex $(\Omega(M),d)$ by tensoring with $S(\mfg^*)^G$. Therefore, we have an $S(\mfg^*)^G$-algebra isomorphism
\begin{equation}\label{eq:H*trivial}
H^*_G(M) \cong S(\mfg^*)^G \otimes H^*(M),
\end{equation}
where $S(\mfg^*)^G$ acts only on the first factor of the right hand side. 
In particular, $H^*_G(M)$ is a free module over $S(\mfg^*)^G$. Particularly important is the case where $M$ consists of a single point: we have $H^*_G({\mathrm{pt}}) = S(\mfg^*)^G$.

Later we will encounter classes of actions for which \eqref{eq:H*trivial} holds, but just as an isomorphism of $S(\mfg^*)^G$-modules.
\end{ex}

One shows directly that any $G$-equivariant map $f\colon M\to N$ between $G$-manifolds $M$ and $N$ induces a pullback homomorphism between the Cartan complexes by $(f^*\omega)(X) = f^*(\omega(X))$ which descends to an $S(\mfg^*)^G$-algebra morphism $f^*\colon H^*_G(N)\to H^*_G(M)$. Then the following lemma follows directly from the definitions:
\begin{lemma}\label{lem:algebrastructurefrommaptopoint} The $S(\mfg^*)^G$-algebra structure $i\colon S(\mfg^*)^G\to H^*_G(M)$ is the same as the map in cohomology induced by the unique map $M\to \{pt\}$.
\end{lemma}

Let us have a look at the zeroth and first equivariant cohomology groups.
\begin{ex}
We have $C_G^0(M) = \Omega^0(M)^G$, the space of $G$-invariant smooth functions $f\colon M\to \RR$. For such $f$, the equivariant differential computes as $d_Gf = df$, and therefore, closed equivariant $0$-forms are locally constant invariant functions. Hence, $H^0_G(M)=H^0(M/G)$ calculates the number of connected components of $M/G$. (In case $G$ is connected, this coincides with the number of connected components of $M$.)
\end{ex}

\begin{ex}\label{ex:equivoneforms} We have $C_G^1(M)=\Omega^1(M)^G$. For $\omega\in \Omega^1(M)^G$, the equivariant differential computes as
\[
(d_G\omega)(X)=d\omega - i_{\overline{X}}\omega
\]
($\omega$ is considered as a constant map $\mfg\to \Omega(M);\, X\mapsto \omega$). Therefore, $d_G\omega=0$ if and only if $d\omega=0$ and $i_{\overline{X}}\omega=0$ for all $X\in \mfg$, i.e., if $\omega$ is a closed basic form. We have computed $C_G^0(M)$ above, which implies that the exact equivariant one-forms are the same as the exact basic one-forms. This shows $H^1_G(M) = H^1_{\bas G}(M)$.
\end{ex}

There is the following relation between basic and equivariant cohomology:

\begin{lemma} The ring homomorphism $\Omega_{\bas G}(M)\to C_G(M);\, \omega\mapsto 1\otimes \omega$ is an inclusion of complexes and therefore defines a homomorphism of $\RR$-algebras $H^*_{\bas G}(M)\to H^*_G(M)$.
\end{lemma}
\begin{proof}
First of all note that $\omega= 1\otimes \omega\in S(\mfg^*)\otimes \Omega(M)$ really is an equivariant differential form because $\omega$ is $G$-invariant. Therefore, the map is well-defined. Clearly, it is an $\RR$-algebra homomorphism. Moreover, we have $d_G(\omega) =  d\omega$ because $\omega$ is horizontal, so it is a map between complexes. 
\end{proof}

\begin{ex}\label{ex:basicequivcohomrelation} 
In general the natural map $H^*_{\bas G}(M)\to H^*_G(M)$ is neither injective nor surjective. Non-surjectivity is clear, as the basic cohomology always vanishes for degrees above the dimension of $M/G$, whereas $H^*_G(M)$ is in general nonzero in infinitely many degrees -- see for instance Example \ref{ex:trivialaction}. In degree $1$, the map is an isomorphism (see Example \ref{ex:equivoneforms}), and in degree $2$ it is always injective: assuming that $\omega = d_G\alpha$, for a closed basic $2$-form $\omega$ and some $\alpha\in C^1_G(M) = \Omega^1(M)^G$, we have
\[
\omega = (d_G\alpha)(X) = d\alpha - i_{\overline{X}}\alpha.
\]
This implies that $i_{\overline{X}}\alpha = 0$ for all $X\in \mfg$, which, together with the $G$-invariance of $\alpha$ says that $\alpha$ is $G$-basic, and thus $d\alpha = \omega$ in $\Omega_{\bas G}(M)$.

The smallest degree in which non-injectivity can occur is $3$, see \cite[Example C.18]{GGK}: consider, on the $4$-sphere
\[
S^4=\{(a,z,w)\mid a^2 + |z|^2 + |w|^2=1\}\subset  \RR\times \CC^2\cong \RR^5
\]
the circle action given by the product of the standard diagonal action on $\CC^2$ and the trivial action on $\RR$. Then one computes (using the equivariant Mayer-Vietoris sequence, Theorem \ref{thm:equivMV} below) that $H^3_{S^1}(S^4)=0$. On the other hand, $H^3_{\bas S^1}(S^4)=\RR$: either using Remark \ref{rem:basicsingularcohom}, by observing that the action is the suspension of the Hopf action on $S^3$, so that the orbit space is homeomorphic to the suspension of $S^2$, which is $S^3$. Alternatively, if one would like to avoid using singular cohomology, one can use basic versions of the Mayer-Vietoris sequence and the homotopy axiom.
\end{ex}

There is also a natural map from equivariant to ordinary de Rham cohomology:
\begin{lemma} The ring homomorphism $\Omega_G(M)\to \Omega(M);\, \omega\mapsto \omega(0)$ is a chain map and therefore defines a homomorphism of $\RR$-algebras $H^*_G(M)\to H^*(M)$.
\end{lemma}
\begin{proof}
We just need to observe that $(d_G\omega)(0) = d(\omega(0)) - i_{\overline{0}} \omega(0) = d(\omega(0))$.
\end{proof}
This map $H^*_G(M)\to H^*(M)$ is in general not injective (for example for trivial actions) and also not surjective (for example for nontrivial free actions). Note that the composition
\[
H^*_{\bas G}(M) \longrightarrow H^*_G(M) \longrightarrow H^*(M)
\]
of the two natural maps just introduced is nothing but the map induced by the inclusion $\Omega_{\bas G}(M)\to \Omega(M)$.
\begin{ex}
Consider an Hamiltonian action of a compact, connected Lie group $G$ on a symplectic manifold $(M,\omega)$. In this situation we have a \emph{momentum map}, i.e., a $G$-equivariant map $\mu\colon M\to \mfg^*$ such that $i_{\overline{X}}\omega = d\mu^X$, where $\mu^X\colon M\to \RR$ is defined by $\mu^X(p) = \mu(p)(X)$.

The momentum map defines an equivariant linear map (which we call $\mu$ again)
\[
\mu\colon \mfg\longrightarrow C^\infty(M);\, X\longmapsto \mu^X.
\]
In particular, $\mu$ can be regarded as an equivariant $2$-form on $M$: $\mu\in (\mfg^*\otimes C^\infty(M))^G \subset C^2_G(M)$. For any element $f\in (\mfg^*\otimes C^\infty(M))^G$ we can consider the equivariant $2$-form $\omega + f$ and compute
\begin{align*}
d_G(\omega+f)(X) &= (d_G\omega)(X) + (d_Gf)(X) \\
&= d\omega - i_{\overline{X}}\omega + df^X + i_{\overline{X}}f^X \\
&= df^X - i_{\overline{X}}\omega.
\end{align*}
This shows that $\omega + f$ is equivariantly closed if and only if $f\in C^2_G(M)$ is a momentum map for the $G$-action. 

In particular, the cohomology class $[\omega]\in H^2(M)$ is in the image of the natural map $H^2_G(M)\to H^2(M)$. It is even true that for any Hamiltonian action on a compact manifold the map $H^*_G(M)\to H^*(M)$ is surjective, see Example \ref{ex:morsebotteqformal} below.
\end{ex}
\section{Locally free actions}
The topic of this section is a theorem of H.\ Cartan \cite{Cartan2} that says that for (locally) free actions, equivariant cohomology is isomorphic to basic cohomology, hence (in the free case) isomorphic to the de Rham cohomology of the orbit space. Recall Remark \ref{rem:basicsingularcohom} which heuristically explained that this is precisely this class of actions for which basic cohomology is a good invariant -- later we will see that equivariant cohomology is a better invariant than basic cohomology for non-free actions.

\begin{defn} We say that an action of a compact Lie group $G$ on a manifold $M$ is \emph{locally free} if all isotropy groups $G_p$ of the action are finite.
\end{defn}

\begin{thm}\label{thm:eqcohomlocallyfreeactions}
For a locally free action of a compact Lie group $G$ on a manifold $M$ the natural map 
\[
H^*_{\bas G}(M)\longrightarrow H^*_G(M)
\] is an isomorphism.
\end{thm}

\begin{proof}
There are many references for a proof of this statement. Besides the original source \cite{Cartan2} one can find it e.g.\ in \cite[Section 5.1]{GuilleminSternberg} or \cite{Nicolaescu}. A generalization to other coefficients can be found in \cite[Section 1.7]{DufloKumarVergne}. We will show the theorem only for the special case $G=S^1$.

The main tool in the proof is the following: because the $S^1$-action is free, $\overline X_p\neq 0$ for all $p\in M$. Thus, we find an $S^1$-invariant $1$-form $\alpha$ on $M$ such that $\alpha(\overline{X}) = 1$. (Choose an $S^1$-invariant Riemannian metric on $M$, and define $\alpha$, for any $p$, to be $1$ on $\overline{X}_p$, and zero on the orthogonal complement of $\overline{X}_p$.)

We first show surjectivity of the map $H^*_{\bas S^1}(M)\to H^*_{S^1}(M)$. Let $\omega\in C_{S^1}^n(M)=\RR[u]\otimes \Omega(M)^{S^1}$ be a closed $S^1$-equivariant differential form on $M$, and write
\[
\omega = \omega_0 + \omega_1 u + \cdots + \omega_k u^k,
\]
where the $\omega_i$ are $S^1$-invariant differential forms, with $\deg \omega_i = n-2i$, and $\omega_k\neq 0$. We assume that $k>0$. Closedness of $\omega$ reads as
\[
0 = d_{S^1}\omega = d\omega_0 + (d\omega_1 - i_{\overline X} \omega_0) u + \cdots + (d\omega_k - i_{\overline X} \omega_{k-1}) u^k - i_{\overline{X}} \omega_k u^{k+1}.
\]
In particular, $i_{\overline X} \omega_k = 0$. We now modify $\omega$ by an exact equivariant differential form:
\begin{align*}
\omega &+ d_{S^1} ((\alpha\wedge \omega_k) u^{k-1}) \\
&= \omega_0 + \omega_1 u + \cdots + (\omega_{k-1} + d(\alpha\wedge \omega_k))u^{k-1} +(\omega_k- i_{\overline X}(\alpha\wedge \omega_k)) u^k \\
&= \omega_0 + \omega_1 u + \cdots + (\omega_{k-1} + d(\alpha\wedge \omega_k))u^{k-1}
\end{align*}
because $i_{\overline X} \alpha=1$ and $i_{\overline X} \omega_k=0$. We have thus found, in the same equivariant cohomology class, a representative with polynomial degree one less. We can continue reducing the degree until we are left with a representative that is an ordinary differential form, which is at the same time equivariantly closed, i.e., closed and basic, and hence also defines an element in $H^n_{\bas S^1}(M)$.

Next, we show injectivity of the map $H^*_{\bas S^1}(M)\to H^*_{S^1}(M)$. So assume that $\eta\in \Omega^n_{\bas S^1}(M)$ is a closed basic form which is equivariantly exact, i.e., there exists $\omega = \omega_0 + \omega_1 u + \cdots + \omega_k u^k$ such that
\[
\eta = d_{S^1} \omega = d\omega_0 + (d\omega_1 - i_{\overline X} \omega_0) u + \cdots (d\omega_k - i_{\overline X} \omega_{k-1}) u^k - i_{\overline{X}} \omega_k u^{k+1}.
\]
Im particular, $\omega_k$ is a basic differential form. If $k>0$, then we reduce the polynomial degree of $\omega$ successively as above, by replacing $\omega$ by $\omega + d_{S^1}((\alpha\wedge \omega_k)u^{k-1})$. Having reduced to the case $k=0$, we are done, because then $d\omega_0 = \eta$, i.e., $\eta$ is exact as a basic differential form.
\end{proof}

Combining this theorem with Proposition \ref{prop:basiccohomfreeaction} we obtain:
\begin{cor}\label{cor:orbitspacedings} For a free action of a compact Lie group $G$ on a manifold $M$ the projection map $M\to M/G$ induces an isomorphism $H^*(M/G)\longrightarrow H^*_G(M)$.
\end{cor}

\begin{rem}
One should note that in the Borel model, see Remark \ref{rem:Borelmodel}, the proof of this theorem is much easier, see e.g., \cite[Section 1.1]{GuilleminSternberg}: to see that $H^*_G(M;R) \cong H^*(M/G;R)$ one only needs to observe that in this case $EG\times_G M\to M/G$ is a fiber bundle with contractible fiber.
\end{rem}

\begin{rem}\label{rem:commutingactionprinciple}
A more general version of this theorem states that for an action of a product $G\times H$ on a manifold $M$ such that the action of the subgroup $G$ is free, the natural map
\[
H^*_H(M/G) \longrightarrow H^*_{G\times H}(M)
\]
is an isomorphism. In Proposition \ref{Appendix:Prop:QuotientShenanigans} we will give a proof of this statement in case $G$ and $H$ are tori.
\end{rem}

\section{Equivariant homotopy and Mayer-Vietoris}
Many standard techniques and results from ordinary cohomology theory have an equivariant counterpart. In this section we prove two of them: the equivariant version of the homotopy axiom and of the Mayer-Vietoris sequence.

\begin{thm} Assume that $G$ acts on $M$ and $N$, and let $f,g\colon M\to N$ be $G$-homotopic equivariant maps, i.e., there exists a smooth $G$-equivariant homotopy $F\colon M\times \RR\to N$ such that $F(\cdot,0)=f$ and $F(\cdot,1)=g$, where we extend the $G$-action to $M\times \RR$ trivially on the second factor. Then $f^*=g^*\colon H^*_G(N)\to H^*_G(M)$.
\end{thm}
\begin{proof}
Recall the usual proof of the homotopy axiom for de Rham cohomology in the nonequivariant setting: one considers the operator 
\[
Q\colon \Omega^k(M\times \RR)\to \Omega^{k-1}(M);\, \alpha\mapsto \int_0^1 i_{\partial_t} \alpha\, dt
\]
and shows that it satisfies the equation
\begin{equation}\label{eq:oldhomotopy}
d\circ Q\circ F^*+Q\circ F^*\circ d = g^*-f^*\colon \Omega(N)\to \Omega(M),
\end{equation}
i.e., that $Q\circ F^*$ is a chain homotopy between $f^*$ and $g^*$, see \cite[\S I.4]{BottTu}, \cite[\S 7.5, Example 9]{Onishchik}. We claim that this equation is still valid equivariantly, in the sense of Equation \eqref{eq:newhomotopy} below. Define
$A\colon C_G^k(M)\to C_G^{k-1}(M)$ by
\[
(A\omega)(X) = Q(F^*(\omega(X))).
\]
First we need to show that $A$ is well-defined, i.e., that $A\omega$ is again a $G$-equivariant differential form. As $F$ is a $G$-homotopy, we have $F(gp,t)=gF(p,t)$ for all $g\in G$, $p\in M$ and $t\in \RR$, i.e., $F\circ g = g \circ F$. Moreover, we have $Q\circ g^* = g^*\circ Q$. Putting this together, we obtain
\[
(A\omega)(\Ad_g X) = Q(F^*(\omega(\Ad_g X))) = Q(F^*((g^{-1})^*(\omega(X)))) = (g^{-1})^* ((A\omega)(X)).
\]
We claim now that 
\begin{equation}\label{eq:newhomotopy}
d_G\circ A + A\circ d_G = g^*-f^*\colon C_G(N)\to C_G(M).
\end{equation}
For any $\omega \in C_G(N)$, we have
\begin{align*}
(d_G(A\omega))(X) &= d((A\omega)(X)) - i_{\overline{X}} ((A\omega)(X)) \\
&=d(Q(F^*(\omega(X)))) - i_{\overline{X}}(Q(F^*(\omega(X))))\\
&=d(Q(F^*(\omega(X)))) + Q(i_{\overline{X}} (F^*(\omega(X))))\\
&=d(Q(F^*(\omega(X)))) + Q(F^*(i_{\overline{X}}(\omega(X)))),
\end{align*}
where we used that $F$ is $G$-equivariant in the last line. Moreover, we have
\begin{align*}
(A(d_G\omega)(X)) &= Q(F^*(d(\omega(X)) - i_{\overline{X}}(\omega(X)))) \\
&= Q(F^*(d(\omega(X)))) - Q(F^*(i_{\overline{X}}(\omega(X)))).
\end{align*}
Adding up these two equations, \eqref{eq:oldhomotopy} implies \eqref{eq:newhomotopy}. This proves the theorem.
\end{proof}
It follows that if $M$ and $N$ are manifolds on which a compact Lie group $G$ acts, and which are $G$-homotopy equivalent, i.e., for which both $f\circ g$ and $g\circ f$ are equivariantly homotopic to the identity map, then $H^*_G(M)$ and $H^*_G(N)$ are isomorphic as graded $S(\mfg^*)^G$-algebras (via the maps $f^*$ and $g^*$).

\begin{thm}[Equivariant Mayer-Vietoris sequence]\label{thm:equivMV} Let $U,V\subset M$ be open $G$-invariant subsets such that $U\cup V=M$. Denote the natural inclusions by $i_U\colon U\to M$, $i_V\colon V\to M$, $j_U\colon U\cap V\to U$, $j_V\colon U\cap V\to V$. Then there is a long exact sequence 
\[
\cdots \longrightarrow H^*_G(M)\overset{i_U^*\oplus i_V^*}{\longrightarrow} H^*_G(U)\oplus H^*_G(V)\overset{j_U^*-j_V^*}{\longrightarrow} H^*_G(U\cap V)\overset{\delta}\longrightarrow H_G^{*+1}(M)\longrightarrow \cdots
\]
The maps $i_U^*$ and $i_V^*$ are $S(\mfg^*)^G$-algebra homomorphisms, but $j_U^*-j_V^*$ and $\delta$ are only $S(\mfg^*)^G$-module homomorphisms.
\end{thm}
\begin{proof} Tensoring the short exact sequence 
\begin{equation}\label{eq:equivmvproof}
0\longrightarrow \Omega(M) \overset{i_U^*\oplus i_V^*}\longrightarrow \Omega(U)\oplus \Omega(V)\overset{j_U^*-j_V^*}\longrightarrow \Omega(U\cap V)\longrightarrow 0
\end{equation}
on the level of differential forms with $S(\mfg^*)$ preserves exactness. We take $G$-invariant forms in each term and  and obtain a sequence
\[
0\longrightarrow C_G^*(M) \overset{i_U^*\oplus i_V^*}\longrightarrow C_G^*(U)\oplus C_G^*(V)\overset{j_U^*-j_V^*}\longrightarrow C_G^*(U\cap V)\longrightarrow 0
\]
of which we need to show exactness. Injectivity at the first term is clear, as well as the inclusion of the image in the kernel at the second term. Let $(\omega,\eta)\in \ker(j_U^*-j_V^*)$. We find $\mu\in S(\mfg^*)\otimes \Omega(M)$ such that $(i_U^*\mu,i_V^*\mu) = (\omega,\eta)$, because the sequence \eqref{eq:equivmvproof}, tensored with $S(\mfg^*)$, is exact. We define $\tilde\mu\in C_G(M)$ as
\[
\tilde\mu = \int_G g^*\mu\, dg,
\]
where $g$ acts on $S(\mfg^*)\otimes \Omega(M)$ diagonally, i.e.,
\[
\tilde\mu(X) = \int_G (g^{-1})^*\mu(\Ad_{g^{-1}}X)\, dg
\]
for $X\in \mfg$, and claim that $(i_U^*\tilde\mu,i_V^*\tilde\mu) = (\omega,\eta)$ as well. For that, we compute
\begin{align*}
i_U^*\tilde\mu(X) &= \int_G i_U^* (g^{-1})^*\tilde \mu(\Ad_{g^{-1}} X)\, dg = \int_G (g^{-1})^* i_U^*\tilde \mu(\Ad_{g^{-1}} X)\, dg\\
&= \int_G (g^{-1})^* \omega(\Ad_{g^{-1}} X)\, dg = \int_G \omega(X)\, dg = \omega(X)
\end{align*}
because $\omega$ is already $G$-invariant. Analogously, $i_V^*\tilde\mu = \eta$, so we have shown exactness at the second term.

For the surjectivity we argue similarly: we start with a possibly noninvariant preimage of an element in $C_G^*(U\cap V)$, and average (both components separately). Thus, we have an induced long exact sequence in equivariant cohomology.
\end{proof}

\begin{ex}\label{ex:S1aufS2}
Consider the $S^1$-action on $S^2$ by rotation around the $z$-axis. Let $S^2=U\cup V$ be the covering of $S^2$ by upper and lower hemisphere. Then $U$ and $V$ are $S^1$-equivariantly homotopy equivalent to the north respectively to the south pole, and $U\cap V$ is $S^1$-equivariantly homotopy equivalent to the equator. Therefore,
\[
H^*_{S^1}(U)=H^*_{S^1}(V)=H^*_{S^1}({\mathrm{pt}})=\RR[u]
\]
with $u$ in degree two, and using Theorem \ref{thm:eqcohomlocallyfreeactions},
\[
H^*_{S^1}(U\cap V) = H^*_{S^1}(S^1) = \RR
\]
concentrated in degree zero. We obtain an exact sequence
\[
\cdots \longrightarrow H^*_{S^1}(S^2)\longrightarrow \RR[u] \oplus \RR[u] \overset{\varphi}\longrightarrow \RR \longrightarrow \cdots
\]
where the map $\varphi$ is given by $\varphi(f,g)=f(0)-g(0)$. It is surjective, so the sequence is in fact short exact and we obtain an isomorphism of $\RR[u]$-algebras
\[
H^*_{S^1}(S^2)=\{(f,g)\in \RR[u]\oplus \RR[u]\mid f(0)=g(0)\}.
\]
Note that $H^*_{S^1}(S^2)$ is a free $\RR[u]$-module: a basis is given by $(1,1)$ and $(u,-u)$.

Note also the peculiar feature of this example that the map on equivariant cohomology induced by the inclusion of the fixed point set into the manifold is injective (the fixed point set is exactly the union of north and south pole). It will be a consequence of the Borel Localization Theorem that this is the case for a large class of actions.
\end{ex}

\section{Equivariant formality} \label{sec:equivariantformality}

Starting with this section, we will make use of the spectral sequence of the Cartan model, which is introduced in Section \ref{Appendix:Section:CartanSpeq}.
\begin{defn}
An action of a compact Lie group $G$ on a smooth manifold $M$ is \emph{equivariantly formal} if the spectral sequence of the Cartan model collapses at the $E_1$-term. 
\end{defn}

\begin{rem}
The term \emph{equivariant formality} was introduced twenty years ago in \cite{GKM}. In the context of the Borel model, see Remark \ref{rem:Borelmodel}, the Serre spectral sequence of the (Borel) fibration $EG\times_G M \to BG$ at and after $E_2$ is equivalent to the spectral sequence of the Cartan model at and after $E_2$. Since $E_1=E_2$ in the Cartan model (see Remark \ref{rem:doddvanishes}, the collapse of the Serre spectral sequence at the $E_2$-term is equivalent to equivariant formality of the action. This collapse is, in turn, equivalent to the surjectivity of the map induced in cohomology by the fiber inclusion (cf.\ Theorem \ref{thm:bigthmequivformal} below), which is usually described by saying that the fiber is \emph{totally nonhomologous to zero}, or that the fibration itself is \emph{totally nonhomologous to zero}, abbreviated \emph{TNHZ}, see e.g.\ \cite{BredonBook}, \cite{AlldayPuppe}, or \cite{FelixOpreaTanre}. Instead of the term \emph{equivariant formality} many authors thus just speak about $M$ being (totally) nonhomologous to zero in the Borel fibration. This condition already appears in \cite[Chapter XII]{Borel}.

It was shown in \cite[Theorem 1.5.2]{GKM} that equivariant formality implies formality properties of certain differential graded modules, which explains the choice of terminology. One might argue though that this nomenclature is not optimal as the formality aspect is just a consequence of the much stronger condition of equivariant formality and there are not many connections to the notion of formality from the point of view of rational homotopy theory. One such connection was given in \cite{CarlsonFok} where the authors prove that if the isotropy action of a homogeneous space is equivariantly formal, then the space is formal. Note that the other implication is not valid, see e.g.\ \cite[Example 4.2]{CarlsonFok}.
\end{rem}
The following theorem collects some equivalent formulations of equivariant formality, as well as some justification of its relevance: Condition $(5)$ says that for equivariantly formal actions the ordinary de Rham cohomology of $M$ is determined by the equivariant cohomology algebra. Note that the equivalence of $(1)$ and $(3)$ is not trivial:  by Proposition \ref{prop:E1term} the $E_1$-term of the spectral sequence is $S(\mfg^*)^G\otimes H^*(M)$, so equivariant formality tells us directly that $H^*_G(M)\cong S(\mfg^*)^G\otimes H^*(M)$, but this isomorphism is only one of graded vector spaces. In general, $H^*_G(M)$ and $E_\infty$ are not isomorphic as $S(\mfg^*)^G$-modules -- see Section \ref{sec:counterexample} for a counterexample. 

\begin{thm}\label{thm:bigthmequivformal} The following conditions are equivalent, for an action of a compact, connected Lie group $G$ on a compact manifold $M$:
\begin{enumerate}
\item The $G$-action is equivariantly formal.
\item The canonical map $H^*_G(M)\to H^*(M)$ is surjective.
\item
There is an isomorphism of graded $S(\mfg^*)^G$-modules
\[
H^*_G(M) \cong S(\mfg^*)^G\otimes H^*(M).
\]
(In particular $H^*_G(M)$ is a free module over $S(\mfg^*)^G$.)
\end{enumerate}
If these conditions are satisfied, then also the following statements hold true:
\begin{enumerate}
\item[(4)] The kernel of the canonical map $H^*_G(M)\to H^*(M)$ is the ideal generated by  the image of $S^+(\mfg^*)^G\to H^*_G(M)$, i.e.,
\[
S^+(\mfg^*)^G\cdot H^*_G(M) = \Big\{\sum_i f_i [\eta_i]\mid f_i\in S^+(\mfg^*)^G,\, [\eta_i]\in H^*_G(M)\Big\}.
\]
Here, $S^+(\mfg^*)^G$ denotes the positive degree elements in $S(\mfg^*)^G$.
\item[(5)] We have an isomorphism of $\RR$-algebras
\begin{equation}\label{eq:ordinarycohomfromeqcohom}
H^*(M) \cong \frac{H^*_G(M)}{S^+(\mfg^*)^G\cdot H^*_G(M)}.
\end{equation}
\end{enumerate}
\end{thm}
\begin{proof}
We first show that $(1)$ and $(2)$ are equivalent. Assuming $(1)$, we consider a cohomology class in $H^n(M)$, represented by a $G$-invariant differential form $\omega_0$. As $d_G\omega_0\in C_G^{2,n-1}(M)$ we have $\omega_0\in A_2^{0,n}$ and can consider the element $[\omega_0]\in E_2^{0,n}$, where we use the notation from Section \ref{Appendix:ConstrSec}. The latter is annihilated by the differential $d_2$, because $d_2\colon E_2\to E_2$ is the zero map by assumption. Thus $d_G\omega_0$ lies in $d_G(A_1^{1,n-1})+A_1^{3,n-2}$. Consequently we find $\omega_1\in C^{1,n-1}_G(M)$ with $d_G\omega_1+d_G\omega_0\in C_G^{3,n-2}(M)$.
Now the element $\omega_0+\omega_1$ lies in $A_3^{0,n}$ and induces an element of $E_3^{0,n}$. Using now that $d_3=0$ we inductively construct an element $\omega=\omega_0+\ldots+\omega_n$ with $d_G\omega=0$ and $\omega(0)=\omega_0$. We have shown that $H^*_G(M)\to H^*(M)$ is surjective.

Assume now that $(2)$ holds, i.e., that we can extend any closed $G$-invariant form $\omega_0$ to a closed equivariant differential form $\omega_0 + \omega_1 + \cdots$. But again by definition of the higher differentials in the spectral sequence this means that all $d_r$, $r=1,2,\ldots$, vanish. (Inductively; first they vanish on $E_r^{0,*}$, but because the $E_r$ are modules over $S(\mfg^*)^G$, and the $d_r$ are $S(\mfg^*)^G$-linear, they vanish completely.) Thus, $(1)$ holds.

We next show that $(2)$ implies $(4)$ and $(5)$. It is clear that $S^+(\mfg^*)^G\cdot H^*_G(M)$ is contained in the kernel of the canonical map $H^*_G(M)\to H^*(M)$. So let $\omega = \omega_0 + \omega_1+\cdots \in H^*_G(M)$ be an element in the kernel, where we use the same notation as above: the index $i$ refers to the polynomial degree of $\omega_i$. Being in the kernel means that $\omega_0=d\beta_0$ is exact as an ordinary invariant differential form. By replacing $\omega$ by $\omega - d_G \beta_0$ we can assume that $\omega_0=0$. Now consider $\omega_1$. Because $d\omega_1=0$, and the $E_1$-term is $S(\mfg^*)^G\otimes H^*(M)$, we can (by adding an appropriate exact form) assume that $\omega_1\in S^1(\mfg^*)^G\otimes \Omega(M)^G$, i.e., $\omega_1 = \sum_j f_j \gamma_j$, for $G$-invariant linear forms $f_j$, and closed $G$-invariant forms $\gamma_j$. Now, because $H^*_G(M)\to H^*(M)$ is surjective, we can extend the $\gamma_j$ to equivariantly closed differential forms $\tilde\gamma_j$, and subtract $\sum_j f_j \tilde\gamma_j$ from $\omega$ to obtain an element in the kernel of the form $\omega_2 + \omega_3 + \cdots$. By continuing in the same way, we have shown the desired expression for the kernel, i.e., $(4)$. Statement $(5)$ follows directly by combining $(2)$ with $(4)$.

Using this implication, we next show that $(1)$ and $(2)$ imply $(3)$: we construct a module isomorphism $H^*_G(M)\cong S(\mfg^*)^G\otimes H^*(M)$. More precisely, we fix a vector space basis $\{[\alpha_i]\}$ of $H^*(M)$, and preimages $[\beta_i]$ of the $[\alpha_i]$ under the canonical map $H^*_G(M)\to H^*(M)$, which exist by $(2)$. In other words, the $\beta_i$ are equivariant differential forms whose polynomial parts are cohomologous to $\alpha_i$. We wish to show that $H^*_G(M)$ is a free $S(\mfg^*)^G$-module with basis $\{[\beta_i]\}$. 

Let us show that the $[\eta_i]$ span $H^*_G(M)$ as a module over $S(\mfg^*)^G$. We proceed by induction on the degree. For degree zero this is true, because $H^0_G(M) = H^0(M)$. So take an arbitrary class $[\omega]\in H^*_G(M)$. We write $[\omega(0)] = \sum_i a_i [\alpha_i]$, for $a_i\in \RR$. By subtracting $\sum_i a_i [\beta_i]$ from $[\omega]$ we thus obtain an element in the kernel of $H^*_G(M)\to H^*(M)$. By $(4)$, this element is a linear combination $\sum_i f_i [\eta_i]$, for some $f_i$ of positive degree. By induction, the $[\eta_i]$ are contained in the span of the $[\beta_i]$, and hence also $[\omega]$.

Finally, we consider the $S(\mfg^*)^G$-module homomorphism
\[
S(\mfg^*)^G\otimes H^*(M) \longrightarrow H^*_G(M)
\]
given by $f\otimes [\alpha_i]\longmapsto f[\beta_i]$. We have shown that it is surjective. But by the collapse of the spectral sequence (condition $(1)$), for every $n$ the degree $n$ part of the left and the right hand side are isomorphic (as abstract vector spaces). Because they are also finite-dimensional (we assumed that $M$ is a compact manifold, and we know also that the polynomial ring $S(\mft^*)^G$ is finite-dimensional in each degree) this map has to be an isomorphism. We have shown $(3)$.

To conclude, we observe that $(3)$ implies $(1)$: if $H^*_G(M)\cong S(\mfg^*)^G\otimes H^*(M)$, then by Proposition \ref{prop:E1term}, $H^*_G(M)$ and the $E_1$-term of the spectral sequence are isomorphic as graded $S(\mfg^*)^G$-modules, and in particular as graded vector spaces. As both vector spaces are finite-dimensional in every degree, this forces all differentials of the spectral sequence to vanish, i.e., the action to be equivariantly formal.
\end{proof}

\begin{rem}
Using more results from the appendix, one can shorten the argument. Without taking the detour through $(4)$ and $(5)$, the equivalent conditions $(1)$ and $(2)$ imply $(3)$ using Lemma \ref{Appendix:Lem:modulegenerators}: a vector space basis of $H^*(M)$ is a module basis of $E_\infty \cong E_1 \cong S(\mfg^*)^G\otimes H^*(M)$, which induces by Lemma \ref{Appendix:Lem:modulegenerators} a set of generators of the $S(\mfg^*)^G$-module $H^*_G(M)$ of the same cardinality. Then the same argument as in the proof above shows that this generating set is in fact a basis.

Having shown in this way that $(1)$, $(2)$ and $(3)$ are equivalent, the implication of $(4)$ and $(5)$ is immediate: $S^+(\mfg^*)^G\otimes H^*(M) \subset S(\mfg^*)^G\otimes H^*(M) \cong H^*_G(M)$ is a subspace of codimension $\dim H^*(M)$, contained in the kernel of the surjection $H^*_G(M)\to H^*(M)$. Thus, $S^+(\mfg^*)^G\cdot H^*_G(M)$ equals the kernel.
\end{rem}

\begin{ex}
Any trivial action is equivariantly formal. For a trivial action, we have $H^*_G(M) = S(\mfg^*)^G\otimes H^*(M)$ even as an algebra over $S(\mfg^*)^G$. 
\end{ex}

\begin{ex} \label{ex:oddcohomzeroeqformal}
More generally, in Corollary \ref{thm:hoddcollapse} we show that the spectral sequence of the action collapses at the $E_1$-term whenever $H^{\odd}(M)$ vanishes. Thus any Lie group action on such a manifold is equivariantly formal.
\end{ex}

\begin{ex}\label{ex:S1aufS2equivformal}
The simplest nontrivial example of an action on a compact manifold with vanishing odd-dimensional cohomology is the standard circle action on the $2$-sphere. In Example \ref{ex:S1aufS2} we identified its equivariant cohomology as
\[
H^*_{S^1}(S^2) \cong \{(f,g)\in \RR[u]\oplus \RR[u]\mid f(0)=g(0)\}.
\]
Any element $(f,g)\in H^*_{S^1}(S^2)$ can be written in the form
\[
(f,g) = \frac12(f+g,f+g) + \frac12(f-g,g-f) = \frac12(f+g)(1,1) + \frac{f-g}{2u} (u,-u),
\]
where we note that because $f(0)=g(0)$, the polynomial $f-g$ is divisible by $u$. Moreover, the elements $(1,1)$ and $(u,-u)$ are linearly independent over $\RR[u]$. Thus, $H^*_{S^1}(S^2)$ is a free module over $\RR[u]$, with basis $\{(1,1),(u,-u)\}$. Note that $H^*(S^2)$ is a graded vector space, with one-dimensional components in degree $0$ and $2$, which are precisely the degrees of the elements $(1,1)$ and $(u,-u)$.

By Theorem \ref{thm:bigthmequivformal} we can recover the ordinary cohomology of $S^2$ from the equivariant one: 
\[
H^*(S^2) \cong \frac{H^*_{S^1}(S^2)}{u\cdot \RR[u]\cdot (1,1) \oplus u\cdot \RR[u] \cdot (u,-u)}
\]
As a vector space, $H^*(S^2)$ is spanned by the cosets of $(1,1)$ and $(u,-u)$. The ring structure is the obvious one, where $(1,1)$ is the unit.

The same argument works in full generality: if one is able to determine a basis $e_1,\ldots,e_k$ of $H^*_G(M)$ as an $S(\mfg^*)^G$-module, for any equivariantly formal $G$-action, then $H^*(M)$ is, as a vector space, isomorphic to the real vector space with the $e_i$ as basis. The multiplicative structure is encoded in the abstract quotient \eqref{eq:ordinarycohomfromeqcohom}.
\end{ex}
\begin{cor}\label{cor:subgroupeqformal}
Consider an equivariantly formal action of a compact, connected Lie group $G$ on a manifold $M$. Then, for any compact, connected Lie subgroup $H\subset G$, the induced $H$-action on $M$ is equivariantly formal as well.
\end{cor}
\begin{proof}
Restiction of an equivariant differential form $\omega\colon \mfg\to \Omega(M)$ to $\mfh$ defines a natural morphism $C_G(M)\to C_H(M)$ which descends to a map $H^*_G(M)\to H^*_H(M)$. Then the statement follows directly from Theorem  \ref{thm:bigthmequivformal} because the canonical map $H^*_G(M)\to H^*(M)$ factors through $H^*_H(M)$.
\end{proof}

Many important classes of actions are equivariantly formal.
\begin{ex}\label{ex:morsebotteqformal} Consider an action of a torus $T$ on a compact manifold $M$. If there exists a $T$-invariant Morse-Bott function $f\colon M\to \RR$ such that the critical set of $f$ is equal to the fixed point set $M^T$, then the action is equivariantly formal. Although not using precisely this formulation, the arguments to show this were given simultaneously by several authors, in  \cite{Duflot}, \cite{AtiyahBott}, \cite{Ginzburg}, and \cite{Kirwan}. Roughly, one shows, using an equivariant Thom isomorphism, that for every critical value $\kappa$ of $f$ one has a short exact sequence
\[
0 \longrightarrow H^*_T(M^{\kappa +\varepsilon},M^{\kappa-\varepsilon}) \longrightarrow H^*_T(M^{\kappa + \varepsilon})\longrightarrow H^*_T(M^{\kappa-\varepsilon})\longrightarrow 0
\]
in (Borel) equivariant cohomology, where for any $a$ we denote the respective sublevel set by $M^a = \{p\in M\mid f(p)\leq a\}$. This implies, inductively, that all $H^*_T(M^a)$ are free $S(\mft^*)$-modules. It was observed in  \cite{GT} that the same argument goes through in the context of Cohen-Macaulay actions, see Section \ref{sec:CohenMacaulay} below, for Morse-Bott functions whose critical set is the union of $b$-dimensional orbits, where $b$ is the lowest occurring orbit dimension.

For example, given any Hamiltonian torus action on a compact symplectic manifold, a generic component of the moment map $\mu\colon M\to \mft^*$ is a Morse-Bott function with this property, thus showing that any Hamiltonian torus action on a compact symplectic manifold is equivariantly formal.
\end{ex}

\begin{ex}
A natural class of actions is given by isotropy actions of homogeneous spaces, i.e., the action of a connected Lie group $H$ on a homogeneous space of the form $G/H$. If $G$ and $H$ are of equal rank, then even the $G$-action on $G/H$ is equivariantly formal, see Theorem \ref{thm:homogeneousspacesequalrank} below, so the $H$-action is, by Corollary \ref{cor:subgroupeqformal}, equivariantly formal as well. (In fact, in this case  $H^{\odd}(G/H)$ vanishes, see again Theorem \ref{thm:homogeneousspacesequalrank}, so that any action on $G/H$ is automatically equivariantly formal.)

In general it is an open question for which homogeneous spaces $G/H$ the isotropy action is equivariantly formal. This question was considered by Shiga and Shiga--Takahashi in \cite{Shiga, ShigaTakahashi}, where they found several sufficient conditions for equivariant formality of isotropy actions (see also \cite[Section 2.1]{Carlson} for a summary of these results). It was shown in the affirmative for symmetric spaces \cite{G}, more generally for spaces such that $H$ is the connected component of the fixed points of any automorphism of $G$ \cite{GH}, and for $\ZZ_2\times \ZZ_2$-symmetric spaces in \cite{Hagh}. Some examples of homogeneous spaces whose isotropy action is not equivariantly formal were given in \cite{ShigaTakahashi} and \cite{Shiga}, and the equivariantly formal homogeneous spaces with $H\cong S^1$ were classified in \cite{Carlson}. In \cite{CarlsonFok} it was shown that equivariant formality of the isotropy action of $G/H$ implies that $G/H$ is formal in the sense of rational homotopy theory.
\end{ex}

\section{Borel localization} \label{sec:borellocalization}
Is this section, as well as the next, we consider only actions of tori on compact manifolds. Recall that for an equivariant smooth map $f\colon N\to M$ between $T$-manifolds, we can consider its induced map $f^*\colon H^*_T(M)\to H^*_T(N)$ in equivariant cohomology. Both its kernel and its cokernel, $\coker f^* = H^*_T(N)/\im f^*$, are naturally $S(\mft^*)$-modules. Our goal in this section is to prove the following theorem (see \cite[Section (1.7)]{GKM} for information on the history of localization  theorems):

\begin{thm}[Borel localization theorem]\label{thm:borellocalization}
Consider, for an action of a torus $T$ on a compact manifold $M$, the restriction map
\[
H^*_T(M)\longrightarrow H^*_T(M^T).
\]
Its cokernel is a torsion module, and its kernel is the torsion submodule of $H^*_T(M)$.
\end{thm}
The proof we give is a version of the proof in \cite[Section 11]{GuilleminSternberg}, somewhat simplified by avoiding the usage of equivariant cohomology with compact support and the notion of support of a module.  Note that there exist far more general versions of the Borel localization theorem, see e.g.~\cite[Chapter 3]{AlldayPuppe} or \cite[Chapter 3, \S 2]{Hsiang}.

Recall the notion of localization from commutative algebra \cite[Chapter 3]{AtiyahMac}. For a multiplicatively closed subset $S$ of a commutative ring with unit $R$ we denote the localized ring by $S^{-1}R$, and the localization of an $R$-module $A$ by $S^{-1}A$. We will need the fact that localization is an exact functor, see \cite[Proposition 3.3]{AtiyahMac}. In case $A$ is a finitely generated module over an integral domain, and $S=R\setminus \{0\}$, the localization $S^{-1}A$ is a finite-dimensional vector space over the field $S^{-1}R$, and we call its dimension the \emph{rank} of $A$, denoted $\rk_R A$. 

With this notion the statement in the Borel localization theorem that both kernel and cokernel of the restriction map are torsion can be reformulated as follows:
\begin{cor} \label{cor:localizedborel}
For any action of a torus $T$ on a compact manifold, the localized map
\[
S^{-1}H^*_T(M) \longrightarrow S^{-1}H^*_T(M^T),
\]
where $S = S(\mft^*)\setminus \{0\}$, is an isomorphism. The rank of the $S(\mft^*)$-module $H^*_T(M)$ is
\[
\rank_{S(\mft^*)} H^*_T(M) = \dim H^*(M^T).
\]
\end{cor}

Before embarking on the proof, we need to calculate the equivariant cohomology of an orbit $Tp=T/T_p$. (Here we consider only tori -- a more general statement about the equivariant cohomology of transitive actions is shown below in Proposition \ref{prop:eqcohomtransitive}.) Let $\mft'\subset \mft$ be a complement of $\mft_p$ in $\mft$ such that $\exp(\mft')$ is a subtorus $T'$ of $T$. Then $S(\mft^*) = S(\mft_p^*)\otimes S(\mft'^*)$. The Cartan complex $C_T(T/T_p)$ can be written as
\[
C_T(T/T_p) = S(\mft_p^*)\otimes S(\mft'^*)\otimes \Omega(T/T_p)^T,
\]
and because $T_p$ acts trivially on all of $T/T_p$, the $T$-invariance of a differential form on $T/T_p$ is equivalent to the $T'$-invariance. Therefore, we have
\[
C_T(T/T_p) = S(\mft_p^*)\otimes (S(\mft'^*)\otimes \Omega(T/T_p)^{T'}).
\]
The equivariant differential $d_T$ on $C_T(T/T_p)$ acts as $d_T = 1\otimes d_{T'}$, because the $T_p$-fundamental vector fields are zero on $T/T_p$. Thus,
\[
H^*_T(T/T_p) = S(\mft_p^*)\otimes H^*_{T'}(T/T_p).
\]
Because the $T'$-action on $T/T_p$ is locally free and transitive, we have $H^*_{T'}(T/T_p) = H^*(\{\mathrm{pt}\}) = \RR$. Thus,
\[
H^*_T(T/T_p) = S(\mft_p^*)
\]
as $S(\mft^*)$-algebras, where the $S(\mft^*)$-algebra structure is induced by the natural restriction $S(\mft^*)\to S(\mft_p^*)$.

In particular, we see that if $\mft_p\neq \mft$ (i.e., if $p$ is not a $T$-fixed point), then $H^*_T(T/T_p)$ is a torsion module: Let $f\in S(\mft^*)$ be a nonzero linear form on $\mft$ that vanishes on $\mft_{p}$; then multiplication with $f$ is the zero map on $H^*_T(T/T_p)$.

\begin{lemma}\label{lem:maptoorbitsupport}
Let $M$ be a (not necessarily compact) manifold that admits a $T$-equivariant map $\varphi\colon M\to Tp$, where $p\in M$ is not a fixed point of the $T$-action. Then $H^*_T(M)$ is a torsion module.
\end{lemma}
\begin{proof}
We consider the maps
\[
M \overset{\varphi}\longrightarrow Tp \longrightarrow \{pt\}.
\]
In equivariant cohomology they induce homomorphisms
\[
S(\mft^*)\longrightarrow H^*_T(Tp)\overset{\varphi^*}\longrightarrow H^*_T(M).
\]
Because of Lemma \ref{lem:algebrastructurefrommaptopoint}, the $S(\mft^*)$-algebra structure of $H^*_T(M)$ is induced from the unique map to a point, which thus factors through $H^*_T(Tp)$. Above, we computed $H^*_T(Tp) \cong S(\mft_p^*)$, where the $S(\mft^*)$-algebra structure is given by the natural restriction map. Every $f\in S(\mft^*)$ with $f|_{\mft_p}=0$ thus annihilates $H^*_T(M)$, because it already defines the zero element in $H^*_T(Tp)$. 
\end{proof}
Any tubular neighborhood $U$ of an orbit $T p$ admits a $T$-equivariant (retraction) map to $Tp$, so Lemma \ref{lem:maptoorbitsupport} applies to any open $T$-invariant subset of $U$.

\begin{proof}[Proof of Theorem \ref{thm:borellocalization}]
The idea of the proof is to use the equivariant Mayer-Vietoris sequence for a cover $M=U\cup V$, where $U$ is a tubular neighborhood of $M^T$, and $V$ an open $T$-invariant subset of $M\setminus M^T$, with the following property: both $V$ and $U\cap V$ can be covered by finitely many $T$-invariant open neighborhoods to which Lemma \ref{lem:maptoorbitsupport} applies, in the sense that they admit an equivariant map to an orbit in $M\setminus M^T$. Let us first construct this covering: we choose two tubular neighborhoods $M^T\subset U'\subset U$ with $\overline{U'}\subset U$. We put $V:=M\setminus \overline{U'}$. As $M\setminus U'$ is compact, it can be covered by finitely many tubular neighborhoods of orbits in $M\setminus M^T$ (none of which intersects $M\setminus M^T$). This finite cover restricts to finite covers of $V$ and $U\cap V$. The open sets in this cover are open subsets of tubular neighborhoods of orbits of points in $M\setminus M^T$, so Lemma \ref{lem:maptoorbitsupport} applies to them.

Now, consider any open subset $W\subset M$ which is a finite union $W = W_1\cup \cdots \cup W_r$ of open $T$-invariant open neighborhoods $W_i$ that admit an equivariant map $f_i\colon W_i\to Tp_i$, where $p_i\in M\setminus M^T$. By Lemma \ref{lem:maptoorbitsupport} we have that $H^*_T(W_i)$ is a torsion module for all $i$. Put $Y_j:=W_1\cup \cdots \cup W_{j-1}$, so that $Y_{j+1} = Y_j\cup W_j$. It follows by induction that 
$H^*_T(Y_j)$ is a torsion module, using the portion
\[
H^*_T(Y_j\cap W_j) \longrightarrow H^*_T(Y_{j+1}) \longrightarrow H^*_T(Y_j) \oplus H^*_T(W_j)
\]
of the equivariant Mayer-Vietoris sequence. Note that we used that with $W_j$ also the intersection $Y_j\cap W_j$ admits an equivariant map to an orbit in $M\setminus M^T$, hence Lemma \ref{lem:maptoorbitsupport} also applies to this set. We have thus shown that $H^*_T(W)$ is a torsion module as well.

This observation in particular applies to the sets $V$ and $U\cap V$ from the open cover $M=U\cup V$ constructed above. Using that $H^*_T(U) \cong H^*_T(M^T)$, the equivariant Mayer-Vietoris sequence of this cover reads
\[
\cdots \longrightarrow H^*_T(U\cap V) \longrightarrow H^*_T(M) \overset{(i^*,j^*)}\longrightarrow H^*_T(M^T) \oplus H^*_T(V) \longrightarrow H^*_T(U\cap V)\longrightarrow \cdots,
\]
where $j\colon V\to M$ is the natural inclusion map. Localizing this exact sequence at $S=S(\mft^*)\setminus \{0\}$, the terms $S^{-1}H^*_T(U\cap V)$ and $S^{-1}H^*_T(V)$ vanish, so that we obtain an isomorphism
\[
S^{-1}H^*_T(M)\longrightarrow S^{-1}H^*_T(M^T)
\]
as in the formulation in Corollary \ref{cor:localizedborel}. That the kernel of the restriction map $H^*_T(M)\to H^*_T(M^T)$ contains the torsion submodule of $H^*_T(M)$ is clear because $H^*_T(M^T)$ is a free module.
\end{proof}

\begin{rem}\label{rem:borelexplanationnofixedpoints} In case the $T$-action has no fixed points, $M^T=\emptyset$. By convention, we understand $H^*_T(\emptyset)=0$.
\end{rem}
\begin{cor} 
$H^*_T(M)$ is a torsion module if and only if the $T$-action has no fixed points.
\end{cor}
\begin{proof} If the $T$-action has no fixed points, then we have just observed that $H^*_T(M)$ is torsion (see Remark \ref{rem:borelexplanationnofixedpoints}). If there are fixed points, then $1\in H^*_T(M)$ is mapped to $1\neq 0\in H^*_T(M^T)$. Because $H^*_T(M^T)$ is a free and hence torsion-free $S(\mft^*)$-module, $1$ is also not a torsion element in $H^*_T(M)$.
\end{proof}

\begin{ex} The Borel localization theorem is wrong without any assumptions on the space acted on. Consider the Borel model (see Remark \ref{rem:Borelmodel}) of the free action of a torus $T$ on the contractible space $ET$. As the projection
\[
ET\times_{T} ET\longrightarrow BT
\]
on the first factor is a homotopy equivalence (it is a fibration with contractible fiber $ET$), the map 
\[
H^*(BT;\RR)\longrightarrow H^*_{T}(ET;\RR)
\]
defining the $H^*(BT;\RR)$-algebra structure is an isomorphism. In particular, the equivariant cohomology $H^*_{T}(ET;\RR)$ is a free $H^*(BT;\RR)$-module although the $T$-action has no fixed points.
\end{ex}

\begin{cor}\label{cor:eqformalinjection}
For an equivariantly formal action of a torus on a compact manifold $M$, the inclusion $M^T\to M$ induces an injective $S(\mft^*)$-algebra homomorphism
\[
H^*_T(M) \longrightarrow H^*_T(M^T) = S(\mft^*)\otimes H^*(M^T).
\]
\end{cor}

One can therefore try to understand the equivariant cohomology of an equivariantly formal action by understanding its image in $H^*_T(M^T)$. 
\begin{ex} We did this already for the standard circle action on $S^2$, with fixed point set the north and south pole $N,S$,  see Example \ref{ex:S1aufS2}, in which we confirmed ad hoc that the inclusion $H^*_{S^1}(S^2)\to H^*(\{N,S\})= \RR[u]\oplus \RR[u]$ is injective, and has as image the $\RR[u]$-subalgebra $\{(f,g)\mid f(0)=g(0)\}$.
\end{ex}
We will give an example with nondiscrete fixed point set below (see Example \ref{ex:inclusionfixedsetconjugation}).

In Example \ref{ex:oddcohomzeroeqformal} we observed that any action on a manifold with vanishing odd-dimensional cohomology is equivariantly formal. If the fixed point set of the torus action is finite, then this is even an equivalence.
\begin{prop}\label{prop:equivformalfinitefixedpointset}
Consider an equivariantly formal action of a torus $T$ on a compact manifold $M$. If the fixed point set of the action is finite, then $H^{\odd}(M) = 0$.
\end{prop}
\begin{proof}
By Corollary \ref{cor:eqformalinjection} we have an injection $H^*_T(M) \to S(\mft^*)\otimes H^*(M^T)$. As $M^T$ is a finite set, $H^*(M^T)$ is concentrated in degree zero. The polynomial ring $S(\mft^*)$ is concentrated in even degrees, so that $H^{\odd}_T(M) = 0$. But by equivariant formality we have
\[
H^*_T(M) \cong S^*(\mft^*)\otimes H^*(M),
\]
so necessarily $H^{\odd}(M)=0$ as well.
\end{proof}

\section{Consequences for the fixed point set}\label{sec:fixedpoints}
Recall that the \emph{Euler characteristic} of a manifold $M$ with finite-dimensional cohomology $H^*(M)$ is defined as
\[
\chi(M):= \dim H^{\even}(M)-\dim H^{\odd}(M).
\]
More generally, one can define the Euler characteristic for any finite-dimensional \emph{$\ZZ_2$-graded vector space} $V$, i.e., a vector space of the form $V=V^{\even}\oplus V^{\odd}$, where we call the elements of $V^{\even}$ and $V^{\odd}$ even and odd elements.
\begin{defn}
Let $V=V^{\even}\oplus V^{\odd}$ be a finite-dimensional $\ZZ_2$-graded vector space. Then the \emph{Euler characteristic} of $V$ is
\[
\chi(V) = \dim V^{\even}-\dim V^{\odd}.
\]
\end{defn}
A fundamental property of the Euler characteristic is that it is preserved under taking cohomology. We omit the (standard) proof.

\begin{lemma}\label{lem:eulercharofcohomology}
Let $V= V^{\even}\oplus V^{\odd}$ be a finite-dimensional vector space over a field $K$, and $d\colon V\to V$ a $K$-linear map that 
\begin{enumerate}
\item is a differential, i.e., $d^2=0$, and 
\item is an \emph{odd} endomorphism, i.e., restricts to maps $d^{\even}:V^{\even}\to V^{\odd}$ and $d^{\odd}\colon V^{\odd}\to V^{\even}$. 
\end{enumerate}
Then
\[
\chi(V) = \chi(H(V,d)),
\]
where $H(V,d) = \ker d/\im d$ (which naturally is a $\ZZ_2$-graded vector space).
\end{lemma}

The following theorem was originally shown by Kobayashi \cite{Kobayashi} without the usage of equivariant cohomology. We present it here as a corollary of the Borel localization theorem.

\begin{thm} \label{thm:eulercharfixedpoints}
Consider the action of a torus $T$ on a compact manifold $M$. Then 
\[
\chi(M) =\chi(M^T).
\]
\end{thm}
\begin{proof}
By Corollary \ref{cor:localizedborel} we have an isomorphism
\[
S^{-1}H^*_T(M) \longrightarrow S^{-1}H^*_T(M^T),
\]
where $S=S(\mft^*)\setminus \{0\}$. The localized equivariant cohomology is not $\ZZ$-graded anymore, but the dichotomy between even and odd degree elements survives after localization. This isomorphism thus restricts to isomorphisms of the respective even and odd parts. As $H^*_T(M^T) = S(\mft^*)\otimes H^*(M^T)$, we therefore have (writing $R=S(\mft^*)$)
\begin{align*}
\chi(M^T) &= \dim_\RR H^{\even}(M^T) - \dim_\RR H^{\odd}(M^T) \\
&= \dim_{S^{-1}R} S^{-1}R\otimes H^{\even}(M^T) - \dim_{S^{-1}R} S^{-1}R \otimes H^{\odd}(M^T)\\
&= \dim_{S^{-1}R} S^{-1}H^{\even}_T(M^T) - \dim_{S^{-1}R} S^{-1}H^{\odd}_T(M^T)\\
&= \dim_{S^{-1}R} S^{-1}H^{\even}_T(M) - \dim_{S^{-1}R} S^{-1}H^{\odd}_T(M).
\end{align*}
We now use the the spectral sequence of the Cartan model to relate this to $\chi(M)$. As observed in Section \ref{Appendix:Section:ModuleStructure} each page $E_r$ of the spectral sequence naturally is an $R$-module, and the differentials are $R$-linear. We now forget the bigrading of the $E_r$, and keep only the total degree. The differential $d_r$, which was of bidegree $(r,-r+1)$, is then an ordinary differential which increases degree by one. Localizing each page of the spectral sequence, we then obtain $\ZZ_2$-graded vector spaces $E_r$, and the differentials $d_r\colon E_r\to E_r$ become odd endomorphisms. Then, each $E_{r+1}$ is the cohomology of $(E_r,d_r)$, in the category of $\ZZ_2$-graded vector spaces. Applying Lemma \ref{lem:eulercharofcohomology} multiple times (noting that there can only be finitely many nontrivial differentials), we compute
\begin{align*}
\chi(M) &=  \dim_\RR H^{\even}(M) - \dim_\RR H^{\odd}(M)\\
&= \dim_{S^{-1}R} S^{-1}R \otimes H^{\even}(M) - \dim_{S^{-1}R} S^{-1}R \otimes H^{\odd}(M)\\
&= \dim_{S^{-1}R} S^{-1}E_1^{\even} - \dim_{S^{-1}R} S^{-1}E_1^{\odd}\\
&= \dim_{S^{-1}R} S^{-1}E_\infty^{\even} - \dim_{S^{-1}R} S^{-1}E_\infty^{\odd},
\end{align*}
where we used Proposition \ref{prop:E1term} for the third equality sign. This equals the result of the first chain of equations above, because the ranks of the even and odd parts of $H^*_T(M)$ and $E_\infty$ agree, as we show in Corollary \ref{Appendix:Cor:RankGleich}.
\end{proof}

\begin{ex}
For any torus action with finitely many fixed points, their number is exactly $\chi(M)$. For example, consider orientable closed surfaces: any nontrivial circle action on the two-sphere has two fixed points, and any nontrivial circle action on the two-dimensional torus has no fixed points at all. Surfaces of higher genus do not admit any nontrivial circle actions.
\end{ex}

\begin{ex} \label{ex:cpnaction} By Example \ref{ex:oddcohomzeroeqformal}, any torus action on a manifold $M$ with $H^{\odd}(M)=0$ is equivariantly formal. For example, this is the case for $\CC P^n$.

As a concrete example, consider the $T^2$-action on $\CC P^2$ given by
\[
(t_0,t_1)\cdot [z_0:z_1:z_2]:=[t_0z_0:t_1z_1:z_2].
\]
Because $\dim H^*(\CC P^2)=3$, we know that if this action has finitely many fixed points, then their number has to be equal to $3$. Indeed, we see that the fixed points are given by
$[1:0:0]$, $[0:1:0]$ and $[0:0:1]$.
\end{ex}

\begin{prop}\label{prop:aeqformalfixedpointset}
For any action of a torus $T$ on a compact manifold $M$, we have $\dim H^*(M^T)\leq \dim H^*(M)$. Moreover, the action is equivariantly formal if and only if $\dim H^*(M^T)=\dim H^*(M)$.
\end{prop}
\begin{proof} By the Borel Localization theorem we have 
\[
\rk H^*_T(M) = \rk H^*_T(M^T) = \rk H^*(M^T)\otimes S(\mft^*) = \dim H^*(M^T). 
\]
On the other hand we know that $\rk H^*_T(M) \leq \dim H^*(M)$: the spectral sequence of the Cartan model has $E_1 = S(\mft^*)\otimes H^*(M)$, which has rank $H^*(M)$. As submodules and quotients of a module cannot have greater rank than the original, we deduce that $\rk(E_\infty)\leq\dim H^*(M)$. Now the first claim follows by Corollary \ref{Appendix:Cor:RankGleich}.

If the action is equivariantly formal, then $H^*_T(M)$ is, as an $S(\mft^*)$-module, isomorphic to $H^*(M)\otimes S(\mft^*)$, hence its rank is equal to $\dim H^*(M)$. If the action is not equivariantly formal, then there exists a nontrivial differential; let $d_r$ be the first of these. As $E_r\cong E_1$ is a free $S(\mft^*)$-module, it follows that $E_{r+1}$ has rank strictly smaller than $E_r$. As by Corollary \ref{Appendix:Cor:RankGleich} the ranks of $H^*_T(M)$ and $E_\infty$ are equal, it follows that $\dim H^*(M^T) = \rk H^*_T(M) = \rk E_\infty < \rk E_1 = \dim H^*(M)$.
\end{proof}

\begin{ex}\label{ex:conjugationeqformal} Consider the action of a compact, connected Lie group $G$ on itself by conjugation. The action, restricted to a maximal torus $T\subset G$ (of dimension $r=\rk G$), has $T$ as fixed point set. Therefore we have $2^r=\dim H^*(G^T)$ as the total dimension of the cohomology of the fixed point set. But on the other hand it is known that also $\dim H^*(G) = 2^r$: A classical theorem of Hopf, see e.g.\ \cite[Theorem 1.3.4]{FelixOpreaTanre}, states that the de Rham cohomology of $G$ is an exterior algebra on generators of odd degree. The fact that the number of generators equals the rank of $G$ can be proven by various means; see \cite[Theorem 3.33]{FelixOpreaTanre} for an argument using rational homotopy theory, or \cite{Fok} for a more elementary argument using the degree of the squaring map $G\to G;\, g\mapsto g^2$. It follows that the $T$-action on $G$ by conjugation is equivariantly formal.
\end{ex}

\begin{ex} \label{ex:inclusionfixedsetconjugation}
Consider, as a special case of Example \ref{ex:conjugationeqformal}, the case $G=\SU(2)$, with maximal torus $S^1\subset \SU(2)$. As the action by conjugation is equivariantly formal, the inclusion $S^1\to \SU(2)$ induces an injection
\[
H^*_{S^1}(\SU(2))\longrightarrow H^*_{S^1}(S^1) = \RR[u]\otimes H^*(S^1).
\]
By equivariant formality we know that, as an $\RR[u]$-module, $H^*_{S^1}(\SU(2))$ is generated by two elements in degree $0$ and $3$. As $H^n_{S^1}(S^1)$ is only one-dimensional for $n=0,3$ (in fact for all $n$), this implies that the restriction map induces an isomorphism of $\RR[u]$-algebras
\[
H^*_{S^1}(\SU(2))\cong \RR[u]\oplus \alpha\cdot u\RR[u],
\]
where $\alpha$ is a generator of $H^1(S^1)$. 
\end{ex}

\begin{cor}\label{cor:actiononfixedseteqformal}
Consider an equivariantly formal action of a torus $T$ on a compact manifold $M$, and $H\subset T$ a subtorus. Then the $T$-action on (every component of) $M^H$ is again equivariantly formal.
\end{cor}
\begin{proof}
By \ref{cor:subgroupeqformal} the subtorus $H$ acts equivariantly formally on $M$. Thus, by Proposition \ref{prop:aeqformalfixedpointset}, $\dim H^*(M^H)=\dim H^*(M)$. Now, the fixed point set of the $T$-action on $M^H$ is again $M^T\subset M^H$, and by equivariant formality of the $T$-action on $M$, we have
\[
\dim H^*(M^T) = \dim H^*(M) = \dim H^*(M^H).
\]
Applying Proposition \ref{prop:aeqformalfixedpointset} again, we conclude that the $T$-action on $M^H$ is equivariantly formal.

Finally, a torus action on a disconnected manifold is equivariantly formal if and only if the action on every connected component is equivariantly formal.
\end{proof}

\begin{ex}
Corollary \ref{cor:actiononfixedseteqformal} in particular says that for an equivariantly formal torus action, every component of a fixed point submanifold $M^H$, where $H\subset T$ is a subtorus, contains a fixed point of the action. Let us give an example of a torus action with fixed points where this property is not satisfied, taken from \cite[Example 2]{Allday}.

Consider $S^1$, embedded in $S^3 = \SU(2)$ as a maximal torus, as well as $S^2 = S^1 \times [0,1]/_\sim$, where we collapse the boundary circles to points. Elements of $S^2$ will thus be written as $[z,t]$, with $z\in S^1$, and $t\in [0,1]$; for $t=0,1$ the elements $[z,t]$ are identical for all $z$. As $S^3$ is simply-connected, we find a homotopy $h\colon S^1\times I\to S^3$ such that $h(z,0)=1$ (the identity element in $S^3$) and $h(z,1)=z$, for all $z\in S^1\subset S^3$.

Define an action of $T^2=S^1\times S^1$ on $M:=S^2\times S^3$ by
\[
(w_1,w_2)\cdot ([z,t],g) := ([zw_1^{-1},t],h(zw_1^{-1},t)w_2h(z,t)^{-1}gw_2^{-1}).
\] 
One directly verifies that this really defines an action.
On the copy of $S^3$ where $t=0$ we have
\[
(w_1,w_2) \cdot ([z,0],g) = ([z,0],w_2gw_2^{-1}),
\]
so the action is conjugation by $w_2$. On the copy of $S^3$ where $t=1$ we have
\[
(w_1,w_2) \cdot ([z,1],g) = ([z,1],zw_1^{-1}w_2z^{-1}gw_2^{-1}) = ([z,1],w_1^{-1}w_2gw_2^{-1}),
\]
so the action is conjugation by $w_2$, followed by left multiplication with $w_1^{-1}$. We picture the whole action as an interpolation between these two actions.

The fixed point set of the full $T$-action is $M^T\cong S^1$, where $S^1$ is the maximal torus in $S^3$ embedded at $t=0$. The restricted action of the subcircle $H = \{(w^2,w)\} \subset T^2$ is given by
\[
(w^2,w)\cdot ([z,t],g) = ([zw^{-2},t],h(zw^{-2},t)wh(z,t)^{-1}gw^{-1}).
\]
For $t\neq 0,1$ there cannot occur any $H$-fixed points, as $zw^{-2}$ cannot equal $z$ for all $w$. For $t=0$ again only the maximal torus is contained in $M^H$. For $t=1$ we have 
\[
(w^2,w)\cdot ([z,1],g) = ([z,1],w^{-1}gw^{-1}),
\]
and because $w^{-1}gw^{-1}=g$ is equivalent to $gwg^{-1}=w^{-1}$ we can only have $w^{-1}gw^{-1}=g$ for all $w\in S^1$ if $g$ is in the normalizer $N_{\SU(2)}(S^1)$. This normalizer is the union $S^1\cup A\cdot S^1$, where $A = \left(\begin{matrix}0 & 1 \\ -1 & 0\end{matrix}\right)$. For elements in the centralizer this equality is not satisfied, but it is satisfied for all elements in $A\cdot S^1$, so we have found another circle in the fixed point set. In total, $M^H$ has two connected components, each of which is diffeomorphic to a circle, and only one of them contains $T$-fixed points.

Concerning equivariant formality, this implies that the $H$-action on $M$ is equivariantly formal (as the total dimension of the cohomology $H^*(M^H)$ is $4$, which is the same as the dimension of $H^*(M)$), but the whole $T$-action is not.
\end{ex}

\section{Cohomology of homogeneous spaces} \label{sec:cohomhomspaces}

In this section we will apply equivariant cohomology theory to obtain information on the cohomology of homogeneous spaces $G/H$, mostly for the case that the ranks of $G$ and $H$ are equal.

\begin{prop}\label{prop:eqcohomtransitive}
Given any two compact, connected Lie groups $H\subset G$, the equivariant cohomology of the $G$-action on $G/H$ by left multiplication is given by
\[
H^*_G(G/H) \cong S(\mfh^*)^H;
\]
its algebra structure $S(\mfg^*)^G\to H^*_G(G/H) = S(\mfh^*)^H$ is given by restriction of polynomials.
\end{prop}
\begin{proof}
Applying Theorem \ref{thm:eqcohomlocallyfreeactions}, or rather the generalization described in Remark \ref{rem:commutingactionprinciple}, twice gives isomorphisms
\[
H^*_G(G/H) \cong H^*_{G\times H}(G) \cong H^*_H({\mathrm{pt}}) = S(\mfh^*)^H
\]
of graded $\RR$-algebras. One needs to confirm that the $S(\mfg^*)^G$-algebra structure is as claimed. To this end, we consider these isomorphisms on the level of equivariant differential forms:
\[
(S(\mfg^*)\otimes \Omega(G/H))^G\longrightarrow (S(\mfg^*)\otimes S(\mfh^*)\otimes \Omega(G))^{G\times H} \longleftarrow S(\mfh^*)^H
\]
where both maps are induced by the natural projection maps. On order to understand where a $G$-invariant polynomial on $\mfg$ is mapped to on the level of cohomology, one needs a chain homotopy inverse of the map on the right, the so-called Cartan map, which is described explicitly in \cite[Theorem 5.2.1]{GuilleminSternberg} or \cite[Section 7]{Meinrenken}. One needs to fix the (in this case unique) connection one-form $\theta$ of the principal $G$-bundle $G\to {\mathrm{pt}}$, which is essentially given by the Maurer-Cartan form of $G$ (but note that $G$ acts by left multiplication on $G$ here). Then, for $Y\in \mfh$ acting on $G$ from the right, we compute
\[
\theta_g(\overline{Y}_g) = \theta_g(dl_g(\overline{Y}_e)) = \theta_g(dr_g(\overline{\Ad_gY}_e)) = -\Ad_gY,
\]
where $l_g$ and $r_g$ denote left and right multiplication with $g\in G$, respectively. Thus, the $H$-equivariant curvature $2$-form $F_H^\theta = d_H\theta + \frac12[\theta,\theta]\in C^2_H(G)\otimes \mfg$ is given by
\[
F_H^\theta(Y)(g) = \Ad_gY,
\]
for every $Y\in \mfh$ and $g\in G$, because $\theta$ satisfies $d\theta + \frac12 [\theta,\theta] = 0$. Thus, for any $G$-invariant polynomial $f\in S(\mfg^*)^G$, replacing the $\mfg$-variable by $F_H^\theta$ is the same as restricting the polynomial to $\mfh$.
\end{proof}
\begin{rem} In \cite[Th\'eor\`eme 24]{DufloKumarVergne}, the proposition is proved under relaxed conditions. Also, just as it is the case with Theorem \ref{thm:eqcohomlocallyfreeactions}, the proof is much easier in the Borel model. We have 
\[
EG\times_G G/H = EG/H = BH,
\]
inducing an isomorphism $H^*_G(G/H) = S(\mfh^*)^H$. When identifying $EG\times_G G/H = BH$, the projection map $EG\times_G G/H\to BG$ becomes the natural map $BH = EG/H \to EG/G = BG$, thus showing the claim about the algebra structure.
\end{rem}

\begin{thm}\label{thm:homogeneousspacesequalrank} For a homogeneous space $G/H$, where $G$ is a compact, connected Lie group and $H\subset G$ a connected closed subgroup, the $G$-action on $G/H$ is equivariantly formal if and only if $\rk G = \rk H$. In this case we have an $\RR$-algebra isomorphism
\[
H^*(G/H) \cong \frac{S(\mfh^*)^H}{(S^+(\mfg^*)^G)}
\]
and $H^*(G/H)$ vanishes in odd degrees.
\end{thm}
\begin{proof}
If the $G$-action is equivariantly formal, then also a maximal torus in $G$ acts in an equivariantly formal fashion, by Corollary \ref{cor:subgroupeqformal}. But the action of a maximal torus in $G$ on $G/H$ by left multiplication can only have fixed points if the ranks of $H$ and $G$ are equal.

Conversely, we consider first the case that $H=T$ is a maximal torus of $G$. In this case $G/T$ admits a CW structure with only even-dimensional cells, by the classical Bruhat decomposition -- see e.g.\ \cite[Section 7]{Mare} (for a nice overview) and references therein, e.g.\ \cite[Theorems 5.1.3 and 5.1.5]{Kumar}. Thus, the odd cohomology of $G/T$ vanishes. By Example \ref{ex:oddcohomzeroeqformal} the $G$-action on $G/T$ is equivariantly formal, and combining the description of the equivariant cohomology in Proposition \ref{prop:eqcohomtransitive} with Theorem \ref{thm:bigthmequivformal} we obtain 
\[
H^*(G/T) \cong \frac{S(\mft^*)}{(S^+(\mfg^*)^G)}.
\]
For a general equal-rank homogeneous space $G/H$ we claim that the fibration
\[
H/T \longrightarrow G/T \longrightarrow G/H
\]
satisfies that the map $H/T\to G/T$ induces a surjection in de Rham cohomology. Indeed, this map is the natural projection
\[
\frac{S(\mft^*)}{(S^+(\mfg^*)^G)}\longrightarrow \frac{S(\mft^*)}{(S^+(\mfh^*)^H)}
\]
which is clearly surjective. Thus, the Leray-Hirsch theorem implies that the cohomology of $G/H$ also vanishes in odd degrees. Thus, in the same way as for $G/T$, the $G$-action on $G/H$ is equivariantly formal, and the desired description of the cohomology of $G/H$ follows.
\end{proof}

\begin{rem} There are various other ways to obtain this theorem, without using the Bruhat decomposition. Given a homogeneous space $G/H$ of equal rank, all isotropy groups of the $G$-action on $H$ have the same rank as that of $G$. For such actions equivariant formality is automatic, see \cite[Proposition 3.7]{GR}. Then, Proposition \ref{prop:eqcohomtransitive} and Theorem \ref{thm:bigthmequivformal} imply the description of the cohomology ring. The vanishing of the odd cohomology then follows directly from the fact $S(\mfh^*)^H$ is concentrated in even degrees, or equally directly from Proposition \ref{prop:equivformalfinitefixedpointset}, because by Lemma \ref{lem:homspacefixedpoints} below, the equivariantly formal action of a maximal torus $T\subset G$ on $G/H$ has finite fixed point set.

Alternatively, one may also argue entirely algebraically and use that $S(\mft^*)$ is a free module over $S(\mfg^*)^G$ (see e.g.\ \cite[Section 18.3]{Kane}) to prove equivariant formality of the $G$-action. 
\end{rem}

\begin{rem}\label{rem:modulestructureG/H}
By Corollary \ref{cor:orbitspacedings} we have, for any connected closed subgroup $H\subset G$ of a compact, connected Lie group $G$ of equal rank, that $H^*_H(G) = H^*(G/H)$, where $H$ acts (freely) on $G$ by right multiplication. We claim that the $S(\mfh^*)^H$-algebra structure of this equivariant cohomology
\[
S(\mfh^*)^H \longrightarrow H^*_H(G) \cong H^*(G/H) \cong \frac{S(\mfh^*)^H}{(S^+(\mfg^*)^G)}
\]
is given by the canonical projection map. To see this, we consider the following commutative diagram, whose upper horizontal isomorphisms are those from the proof of Proposition \ref{prop:eqcohomtransitive}, and whose vertical maps are given by restriction of the acting group: 
\[\xymatrix{
S(\mfh^*)^H \ar[rrd] \ar[r]^{\cong} & H^*_H({\mathrm{pt}}) \ar[r]^{\cong} & H^*_{G\times H}(G) \ar[r]^{\cong} \ar[d] & H^*_G(G/H) \ar[d]\\
& & H^*_H(G) \ar[r]^\cong & H^*(G/H) \ar[r]^\cong & \frac{S(\mfh^*)^H}{(S^+(\mfg^*)^G)}
}\]
Note that the square in the middle commutes because the inverses of the two horizontal maps are induced by the canonical projection $G\to G/H$. The claim follows because traversing the diagram from the top left to the bottom right via the upper path results in the canonical projection map.
\end{rem}

\begin{cor}\label{cor:homspaceequivalences}
Consider a homogeneous space $G/H$, where $H\subset G$ are compact, connected Lie groups. Then $\chi(G/H)\geq 0$. Moreover, the following conditions are equivalent:
\begin{itemize}
\item $\rank G = \rank H$
\item $\chi(G/H) >0$
\item $H^{\odd}(G/H)=0$.
\end{itemize}
\end{cor}
\begin{proof}
In Theorem \ref{thm:homogeneousspacesequalrank} we showed that for homogeneous spaces with $\rk G = \rk H$ the odd degree cohomology vanishes, and hence also the Euler characteristic is positive.

Let us show that whenever $\rk G > \rk H$ the Euler characteristic is zero. Then, as we always have cohomology in degree zero, the odd cohomology cannot vanish either. To see this, we construct a circle action on $G/H$ without fixed points, and apply Theorem \ref{thm:eulercharfixedpoints}: We choose a maximal torus $T_H\subset H$, as well as a maximal torus $T_G\subset G$ containing $T_H$. We can choose a circle $S^1\subset T_G$ which is not $G$-conjugate to a subgroup of $H$. (If this was not the case, then choose a sequence of subcircles $\{\exp(tX_n)\}$, with $X_n\to X\in \mfg$, such that $\{\exp(tX)\}$ is dense in $G$. If there existed $g_n$ such that $\Ad_{g_n}X_n\in \mfh$, then we could find a subsequence, converging to $g\in G$, and this element would satisfy $\Ad_gX\in \mfh$. But then, by continuity, $gGg^{-1}\subset H$, a contradiction.) Then, this circle cannot fix any point $gH\in G/H$, as the $G$-isotropy of this point is $gHg^{-1}$ -- if it fixed $gH$, then it would be conjugate to a subgroup of $H$. 

We thus have found a circle action without fixed points, which shows that the Euler characteristic is zero. 
\end{proof}

We now neglect the ring structure of the cohomology of equal-rank homogeneous spaces obtained in Theorem \ref{thm:homogeneousspacesequalrank}, and concentrate on their Betti numbers. We first obtain a formula for the total Betti number in Proposition \ref{prop:dimGH}, and then describe explicitly the Poincar\'e polynomials in Proposition \ref{prop:homspaceequalrankbetti}.

\begin{lemma}\label{lem:homspacefixedpoints}
Consider a homogeneous space $G/H$, where $H$ and $G$ are compact, connected Lie groups of equal rank, and $T\subset H$ a maximal torus. Then the inclusion $N_G(T)\to G$ induces an injection 
\[
W(G)/W(H) \cong N_G(T)/N_H(T) \longrightarrow G/H
\]
whose image is precisely the fixed point set of the $T$-action on $G/H$.
\end{lemma}
\begin{proof}
We observe that an element $gH\in G/H$ is fixed by $T$ if and only if $g^{-1}Tg\subset H$, i.e., by the conjugacy of maximal tori in $H$, if and only if there exists $h\in H$ such that $h^{-1}g^{-1}Tgh = T$. As $ghH = gH$, this means that the $T$-fixed point set is precisely the image of the composition $N_G(T) \to G \to G/H$ of the natural inclusion with the natural projection.
\end{proof}
\begin{prop}\label{prop:dimGH} For an equal-rank homogeneous space $G/H$, we have
\begin{equation}\label{eq:dimcohomGH}
\dim H^*(G/H) = \frac{|W(G)|}{|W(H)|}.
\end{equation}
\end{prop}
\begin{proof}
This follows from Proposition \ref{prop:aeqformalfixedpointset} because the action of a maximal torus $T\subset H$ is equivariantly formal and has precisely $\frac{|W(G)|}{|W(H)|}$ fixed points.
\end{proof}
\begin{rem}
The equality $\dim H^*(G/T)=|W(G)|$ follows also because the CW structure on $G/T$ given by the Bruhat decomposition has precisely $|W(G)|$ cells. Proposition \ref{prop:dimGH}  is then immediate from the observations on the fibration $H/T\to G/T\to G/H$ given in the proof of Theorem \ref{thm:homogeneousspacesequalrank}.
\end{rem}

\begin{ex}\label{ex:totalBettiGrass}
For the complex Grassmannian of $k$-planes in $\CC^n$

\[
{\mathrm{Gr}}_k(\CC^n) = \quotientmed{\U(n)}{\U(k)\times \U(n-k)}
\] 
we obtain
\[
\dim H^*({\mathrm{Gr}}_k(\CC^n)) = \frac{|W(U(n))|}{|W(U(k))| \cdot |W(U(n-k))|} = \frac{n!}{k!(n-k)!} = {n \choose k}.
\]
\end{ex}
\begin{prop}\label{prop:homspaceequalrankbetti}
Consider a homogeneous space $G/H$ of compact, connected Lie groups $H\subset G$ of equal rank $r$. If 
\[
S(\mfg^*)^G \cong \RR[\sigma_1,\ldots,\sigma_r]
\] 
and
\[
S(\mfh^*)^H \cong \RR[\psi_1,\ldots,\psi_r]
\]
with $\deg \sigma_i = p_i$ and $\deg \psi_i = q_i$ (usual degree of polynomials), then 
\[
P_t(H^*(G/H)) = \prod_{i=1}^r \frac{1-t^{2p_i}}{1-t^{2q_i}},
\]
where $P_t(H^*(G/H)) = \sum_{n=0}^{\dim G/H} b_n(G/H) t^n$ is the Poincar\'e polynomial of $G/H$.
\end{prop}

\begin{proof}
In Theorem \ref{thm:homogeneousspacesequalrank} we observed that the transitive $G$-action on $G/H$ is equivariantly formal. Using Proposition \ref{prop:eqcohomtransitive} and Theorem \ref{thm:bigthmequivformal} we conclude that
\begin{equation}\label{eq:freemoduleG/T}
S(\mfh^*)^H\cong H^*_G(G/H) \cong S(\mfg^*)^G \otimes H^*(G/H);
\end{equation}
here we need these isomorphisms only as one of graded vector spaces (but note that as elements of equivariant cohomology, the $\sigma_i$ and $\psi_i$ have twice the degree they inherited from the polynomial rings). This equality helps to compute the Betti numbers of $G/H$: the Poincar\'e series of $S(\mfh^*)^H$ and $S(\mfg^*)^H$ (for a graded vector space $V = \bigoplus_{n\geq 0} V_n$ with $\dim V_n<\infty$ for all $n$, this is the formal power series $\sum_{n=0}^\infty t^n \dim V_n$) are 
\[
P_t(S(\mfh^*)^H) = \prod_{i=1}^r \frac{1}{(1-t^{2q_i})},\qquad P_t(S(\mfg^*)^G) = \prod_{i=1}^r \frac{1}{(1-t^{2p_i})}.
\]
Then \eqref{eq:freemoduleG/T} implies that 
\[
P_t(S(\mfh^*)^H) = P_t(S(\mfg^*)^G)\cdot P_t(H^*(G/H)),
\]
so that 
\[
P_t(H^*(G/H)) = \prod_{i=1}^r \frac{1-t^{2p_i}}{1-t^{2q_i}}.
\]
\end{proof}

\begin{ex}\label{ex:bettiG/T}
In the special case that $H=T$ is a maximal torus of $G$, the cohomology $H^*(G/T)$ is, as an $\RR$-algebra, generated by the elements in $H^2(G/T)$. The Poincar\'e polynomial is 
\[
P_t(H^*(G/T)) = \prod_{i=1}^r \frac{1-t^{2p_i}}{1-t^2} = \prod_{i=1}^r (1+t^2+t^4 \cdots + t^{2p_i-2}).
\]
In particular, the total Betti number of $G/T$ is
\[
\dim H^*(G/T) = P_1(H^*(G/T)) = \prod_{i=1}^r p_i.
\]
Comparing this with Equation \eqref{eq:dimcohomGH}, i.e., $\dim H^*(G/T) = |W(G)|$, we obtain the following general formula for the order of the Weyl group of $G$ in terms of the generators of the cohomology of $G$:
\[
|W(G)| = \prod_{i=1}^r p_i.
\]
\end{ex}
\begin{ex}
Consider the complex Grassmannian ${\mathrm{Gr}}_k(\CC^n)$ of $k$-planes in $\CC^n$ as in Example \ref{ex:totalBettiGrass}.  In Example \ref{ex:invpolyUn} we computed that for $G=\U(n)$ we have $S(\mfg^*)^G = \RR[\sigma_1,\ldots,\sigma_n]$, where $\deg \sigma_i = i$. Thus, Proposition \ref{prop:homspaceequalrankbetti} gives
\begin{align*}
P_t({\mathrm{Gr}}_k(\CC^n)) &= \frac{(1-t^2)\cdots (1-t^{2n})}{(1-t^2)\cdots (1-t^{2k})(1-t^2)\cdots (1-t^{2(n-k)})}\\
&=\frac{(1-t^{2k+2})\cdots (1-t^{2n})}{(1-t^2)\cdots (1-t^{2(n-k)})}.
\end{align*}
\end{ex}
For more information on the cohomology of homogeneous spaces $G/H$, where $\rk G > \rk H$, we only refer to the literature, e.g.\ \cite{GHV}.

\section{Computing $H^*(M)$ via $H^*_T(M)$} \label{sec:HMHTM}

In Theorem \ref{thm:bigthmequivformal} we have seen that for an equivariantly formal $G$-action on $M$ we have an isomorphism of $\RR$-algebras
\[
H^*(M) \cong \frac{H^*_G(M)}{S^+(\mfg^*)^G\cdot H^*_G(M)}.
\]
This means that whenever we know the equivariant cohomology $H^*_G(M)$ as an $S(\mfg^*)^G$-algebra, we can use this isomorphism to compute the ordinary cohomology $H^*(M)$.

For an equivariantly formal torus action, the Borel localization theorem \ref{thm:borellocalization} states that the restriction map
\[
H^*_T(M) \longrightarrow H^*_T(M^T) = S(\mft^*)\otimes H^*(M^T)
\]
is injective, so one can try to compute $H^*_T(M)$ by understanding its image under this map. This is achieved by the Chang--Skjelbred Lemma, which describes the image only in terms of the $1$-skeleton $M_1 := \{p\in M\mid \dim T\cdot p\leq 1\}$ of the action, see \cite[Lemma 2.3]{ChangSkjelbred}. The original formulation used the Borel model; as $M_1$ is not a manifold, the formulation in terms of the Cartan model reads slightly differently -- see \cite[Section 11.5]{GuilleminSternberg} for the proof.
\begin{thm}[Chang-Skjelbred Lemma]\label{thm:changskjelbred}
The image of the natural restriction map $i^*\colon H^*_T(M) \to H^*_T(M^T)$ is given by
\begin{equation}\label{eq:changintersection}
\bigcap_{H\subset T} i_H^*(H^*_T(M^H)),
\end{equation}
where $H$ runs through all codimension-one subtori of $T$ and $i_H\colon M^T\to M^H$ is the inclusion.
\end{thm}
Note that for almost all codimension-one subtori $H\subset T$ we have $M^H = M^T$; these $H$ are irrelevant for the intersection. The only relevant groups $H$ are the connected components of those isotropy groups of the $T$-action that are of codimension one -- of these there are only finitely many. The one-skeleton $M_1$ of the action is the union of all the $M^H$, where $H$ runs through the codimension-one subtori as above.
\begin{ex}\label{ex:oneskeletoncp2}
Consider the $T^2$-action on $\CC P^2$ from Example \ref{ex:cpnaction}. The orbit space of this action is a triangle. The one-skeleton of the action is the preimage of the boundary of this triangle under the projection to the orbit space. It is the union of three $2$-spheres, any two of which meet in a single point.
\end{ex}

One important special case in which this theorem yields explicitly computable results is that of so-called GKM actions, named after a paper by Goresky, Kottwitz, and MacPherson \cite{GKM}. There, one assumes that the structure of the one-skeleton is as simple as possible: 
\begin{defn}\label{defn:GKM}
We call an action of a torus $T$ on a compact, connected manifold $M$ a \emph{GKM action} if the following conditions are satisfied:
\begin{enumerate}
\item The action is equivariantly formal.
\item The fixed point set of the action is finite.
\item The one-skeleton $M_1$ is a finite union of $T$-invariant two-spheres. 
\end{enumerate}
\end{defn}
Given the second condition, we know that the first one is equivalent to demanding that the odd cohomology groups of $M$ vanish, see Proposition \ref{prop:equivformalfinitefixedpointset}. Easy examples of GKM actions are the standard circle action on $S^2$, or the $T^2$-action on $\CC P^2$ (see Example \ref{ex:oneskeletoncp2}). These can be generalized to the following class of examples:
\begin{ex}
All toric symplectic manifolds are GKM. Indeed, toric symplectic manifolds have vanishing odd cohomology groups \cite[Theorem VII.3.5]{Audin} and finite fixed point set, and at each fixed point the weights of the isotropy representation form a basis of $\mft^*$: if $M$ is $2n$-dimensional, then there are precisely $n$ weights of the isotropy representation at any given fixed point, which have to be linearly independent, as otherwise the common kernel of the weights would determine a positive-dimensional subtorus acting trivially on $M$.
\end{ex}

Let $p\in M^T$ be a fixed point of a GKM action. Then the isotropy representation at $p$ decomposes into two-dimensional irreducible subrepresentations. If $\alpha$ is a weight of the isotropy representation -- which is a linear form on $\mft$, well-defined up to sign -- with weight space $V_\alpha$, and $T_\alpha\subset T$ the subtorus with Lie algebra $\ker \alpha$, then $V_\alpha$ is tangent to $M^{T_\alpha}\subset M_1$. The condition that $M_1$ is a finite union of two-dimensional submanifolds, is equivalent to the condition that the weights of the isotropy representation, at any fixed point, are pairwise linearly independent. Thus, for a GKM action on a manifold of dimension $2n$, in any given fixed point there meet precisely $n$ invariant two-spheres.

To any GKM action one associates, as follows, a labelled graph $\Gamma$, called the \emph{GKM graph} of the action: the vertices $V(\Gamma)$ are given by the fixed points of the action, and we draw an edge (i.e., an element of the edge set $E(\Gamma)$) for any invariant $2$-sphere connecting two fixed points. The argument above shows that this graph, for $M$ of dimension $2n$, is $n$-valent. Additionally, we label the edge as follows: the tangent space of an invariant two-sphere in one of the two fixed points is a two-dimensional invariant submodule of the isotropy representation, and there is a codimension-one subtorus $H\subset T$ that acts trivially on it. We put any nonzero linear form $\alpha\in \mft^*$ that vanishes on $\mfh$  as a label of the corresponding edge.
\begin{ex} A classical result of Atiyah \cite{Atiyah2} and Guillemin--Sternberg \cite{GuiSte} states that the image of the momentum map $\mu\colon M\to \mft^*$ of an Hamiltonian torus action on a symplectic manifold $M$ is a convex polytope. For a toric symplectic manifold $M$, the dimension of an orbit $T\cdot p$ is precisely the smallest dimension of a face containing $\mu(p)$. It follows that the GKM graph of a toric symplectic manifold is precisely the one-skeleton of the polytope $\mu(M)$.
\end{ex}

\begin{ex}
Consider a homogeneous space $G/H$, with $\rk G = \rk H$, equipped with the action of a maximal torus $T\subset H$ by left multiplication. We showed in Section \ref{sec:cohomhomspaces} that this action is equivariantly formal, and that the fixed point set of this action is given by the finite set $W(G)/W(H)$. In \cite{GuilleminHolmZara} it was observed that the $T$-action is GKM, and the GKM graph was determined explicitly in terms of the root systems of $G$ and $H$ (see \cite[Theorem 2.4]{GuilleminHolmZara}).
\end{ex}

The equivariant cohomology of a GKM action is encoded in the GKM graph:

\begin{thm} Consider a GKM action of a torus $T$ on a compact, connected orientable manifold $M$. Then 
\begin{align*}
&H^*_T(M) \cong \Big\{(f_p)\in \bigoplus_{p\in M^T} S(\mft^*) \Bigm|  f_p-f_q \in (\alpha) \text{ if there is an edge from } p \text{ to } q \text{ labelled } \alpha\Big\}.
\end{align*}
Here, $(\alpha)$ denotes the principal ideal generated by $\alpha$.
\end{thm}
\begin{proof}
By Theorem \ref{thm:changskjelbred} the image of the (injective) natural restriction map $H^*_T(M)\to H^*_T(M^T)$ is 
\[
\bigcap_H i_H^*(H^*_T(M^H)),
\]
where $H$ runs through the codimension one subgroups of $T$, and $i_H\colon M^T\to M^H$ is the inclusion. As observed before, under our assumptions each component $N$ of one of the $M^H$ is either a single fixed point or a two-sphere $S^2$ with an action of $T/H \cong S^1$. To compute the equivariant cohomology of $N$, we generalize Example \ref{ex:S1aufS2} slightly: take $N = U\cup V$, where $U$ and $V$ are $T$-equivariantly homotopy equivalent to a fixed point. (Modulo the ineffective kernel, $N$ is equivariantly diffeomorphic to $S^2$ with the standard circle action. This follows from the theory of cohomogeneity-one actions, as such actions are determined by their group diagram.) So $H^*_T(U) = H^*_T(V) = S(\mft^*)$. Moreover, $U\cap V$ is homotopy equivalent to an invariant circle, whose isotropy Lie algebra is $\mfh$, so $H^*_T(U\cap V) = S(\mfh^*)$.  We thus obtain an exact sequence
\[
\cdots \longrightarrow H^*_T(N) \longrightarrow S(\mft^*)\oplus S(\mft^*) \overset{\varphi}\longrightarrow S(\mfh^*)\longrightarrow \cdots,
\]
where the map $\varphi$ is given by $\varphi(f,g) = f|_\mfh - g|_\mfh$. We thus obtain that
\[
H^*_T(N) \cong \{(f,g)\in S(\mft^*)\oplus S(\mft^*)\mid f|_\mfh = g|_\mfh\}.
\]
Now, the condition that $f|_\mfh =g|_\mfh$ is equivalent to the condition that the polynomial $f-g$ is in the kernel of the restriction map $S(\mft^*)\to S(\mfh^*)$. This kernel is a principal ideal, generated by any nonzero linear form that vanishes on $\mfh$. This is precisely the relation prescribed by the edge corresponding to $N$. 
\end{proof}

\begin{ex}\label{ex:cp2gkm}
Consider the action of $T^2$ on $\CC P^2$. We already understand the one-skeleton of the action, which consists of three invariant two-spheres. They are given by $\{[z:w:0]\}$, $\{[z:0:w]\}$ and $\{[0:z:w]\}$, whose isotropy groups are $\{(t,t)\mid t\in S^1\}$, $\{1\}\times S^1$, and $S^1\times \{1\}$, respectively. Choosing $\{u,v\}$ as the dual basis to the standard basis of $\mft\cong \RR^2$, the labels of the graph (which is a triangle) are given by $u$, $v$, and $u-v$.
  \begin{figure}[htb]
 \includegraphics[width=117pt]{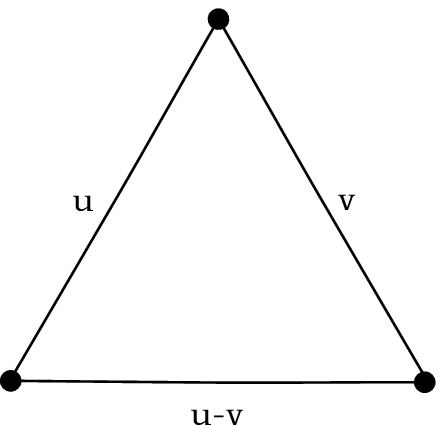}
 \caption{GKM graph of $\CC P^2$}
 \end{figure}
The equivariant cohomology is thus given by
\[
H^*_{T^2}(\CC P^2) \cong \Big\{(f,g,h)\in \RR[u,v]^3\Bigm| f-g\in (u),\, f-h\in (v),\, g-h\in (u-v)\Big\},
\]
with the $S(\mft^*)$-algebra structure induced from the equivariant cohomology of the fixed point set, i.e., componentwise multiplication.

From this, we can now determine the graded ring structure of the ordinary cohomology of $\CC P^2$. One checks that 
\[
(1,1,1),\qquad (v,v-u,0),\qquad (uv,0,0)
\]
are $\RR[u,v]$-module generators of the equivariant cohomology (which have degree $0,2,4$ as predicted by Theorem \ref{thm:bigthmequivformal}). To understand the ring structure we have to multiply 
\begin{align*}
(v,v-u,0)\cdot (v,v-u,0) &\equiv (v^2,(v-u)^2,0) \\
&\equiv (v^2,v^2-2uv+u^2,0) - v(v,v-u,0) + u(v,v-u,0)\\
&\equiv (uv,0,0)
\end{align*}
where we compute modulo $S^+(\mft^*)\cdot H^*_{T^2}(\CC P^2)$, i.e., in the quotient $H^*_{T^2}(\CC P^2)/S^+(\mft^*)\cdot H^*_{T^2}(\CC P^2)$. Also, $(v,v-u,0)^3 \equiv 0$. It follows that 
\[
H^*(\CC P^2) \cong \RR[\omega]/(\omega^3),
\]
where $\omega$ is of degree $2$ (which we of course knew before).
\end{ex}
A detailed introduction to GKM theory with many explicit computations can be found in \cite{Tymoczko}.

One can not only apply GKM theory to concrete computations, but also to obtain structural results on certain classes of actions. For instance, in \cite{GW} it was shown that all known examples of even-dimensional positively curved Riemannian manifolds admit isometric GKM actions, and described their GKM graphs. The graphs that occur are simplices and the complete bipartite graph $K_{3,3}$, with possibly all edges doubled or quadrupled. As an example, see Figure \ref{figrp} (which is taken from \cite{GW}) for the GKM graph of the action of the maximal torus of ${\mathrm{Spin}}(8)$ on $F_4/{\mathrm{Spin}}(8)$ by left multiplication.
  \begin{figure}[htb]
 \includegraphics[width=117pt]{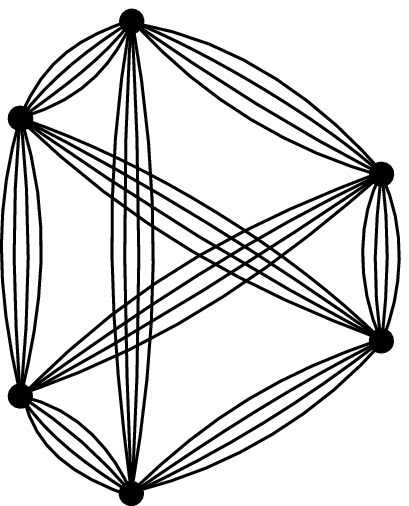}
 \caption{GKM graph of $F_4/{\mathrm{Spin}}(8)$}
 \label{figrp}
 \end{figure}
Restricting to GKM$_3$-actions (i.e., actions for which the two-skeleton of the action is the union of four-dimensional submanifolds) one obtains the following theorem.
\begin{thm}
Let $M$ be a compact, connected, positively curved, orientable Riemannian manifold. If $M$ admits an isometric GKM$_3$ torus action, then $M$ has the real cohomology ring of a compact rank one symmetric space.
\end{thm}
To prove this theorem we determined all possible GKM graphs under the given curvature assumption, using the classification of four-dimensional positively curved $T^2$-manifolds by Grove and Searle \cite{GroveSearle}.

Finally, we mention that GKM theory allows for various generalizations. One possibility to generalize is to allow a nonisolated fixed point set. This was considered in the context of Hamiltonian actions on symplectic manifolds \cite{GuilleminHolm}, and for equivariantly formal torus actions with one-dimensional fixed point set \cite{Chen}. In He's paper an important feature of the class of actions he considers is that the one-skeleton of the action is the union of (three-dimensional) submanifolds each containing an arbitrary number of fixed point components, contrary to the the classical case in which the invariant two-spheres always contain exactly two fixed points. GKM theory for actions without fixed points was considered in \cite{GNT}, for a certain class of Cohen-Macaulay torus actions (see Section \ref{sec:CohenMacaulay} below). Instead of the one-skeleton of the action one describes the equivariant cohomology of the action in terms of the $b+1$-skeleton $M_{b+1}$ of the action, where $b$ is the lowest occurring dimension of an orbit. The class of actions considered in \cite{GNT} has the property that $M_{b+1}$ is the union of submanifolds, each containing exactly two components of $M_b$. It is also possible to generalize GKM theory to actions of arbitrary compact Lie groups \cite{GM}, as well as to possibly infinite-dimensional equivariant cell complexes \cite{HHH}. One can also abstract from torus actions on manifolds and consider GKM graphs as objects of independent interest (see e.g.\ \cite{GuilleminZara}).  

\section{Algebraic generalizations of equivariant formality} \label{sec:CohenMacaulay}

An important property of equivariant formality of a torus action is that the restriction map
\begin{equation}\label{eq:restrictionmapCM}
H^*_T(M) \longrightarrow H^*_T(M^T)
\end{equation}
is injective. Because the kernel of this map is the torsion submodule by the Borel Localization Theorem \ref{thm:borellocalization}, this property is in fact equivalent not to the freeness of $H^*_T(M)$ but to its torsion-freeness. One can therefore ask the question how different equivariantly formal actions are from actions whose equivariant cohomology is torsion-free.

It was shown in \cite{Allday} that for smooth actions of at most two-dimensional tori, torsion-freeness of the equivariant cohomology is equivalent to equivariant formality. The first example of a non-equivariantly formal torus action whose equivariant cohomology is torsion-free was given in \cite{FranzPuppe}. 

Recently, Allday--Franz--Puppe interpolated between torsion-freeness and freeness of the equivariant cohomology, by using the notion of syzygies \cite{AlldayFranzPuppe}: already Atiyah \cite[Lecture 7]{Atiyah} and Bredon \cite[Main Lemma]{Bredon} observed that equivariantly formal actions satisfy a stronger property than the Chang-Skjelbred Lemma, Theorem \ref{thm:changskjelbred}, namely the exactness of the so-called \emph{Atiyah-Bredon sequence}
\[
0 \to H^*_T(M) \to H^*_T(M^T) \to H^{*+1}_T(M_1,M^T) \to \cdots \to H^{*+k}_T(M_k,M_{k-1})\to 0,
\]
where $M_i$ is the union of the $T$-orbits of dimension at most $i$. Here, we use relative equivariant cohomology in the Borel model (cf.\ Remark \ref{rem:Borelmodel}) to give meaning to the cohomologies occurring in the sequence. In \cite{FranzPuppe25} it was shown that exactness of this sequence is even equivalent to equivariant formality. More precise information was given in \cite{AlldayFranzPuppe}, where the authors showed that exactness of this sequence at the first $i$ positions is equivalent to $H^*_T(M)$ being an $i$th syzygy. Examples of torus actions whose equivariant cohomologies vary among all possible syzygy orders are given by so-called big polygon spaces \cite{FranzBig}.

A different way in which one can generalize the notion of equivariant formality is that of a Cohen-Macaulay action, introduced in \cite{GT}. The relevance of the Cohen-Macaulay property was already observed in \cite{Atiyah}.
\begin{defn} We say that an action of a compact Lie group $G$ on a compact manifold $M$ is \emph{Cohen-Macaulay} if $H^*_G(M)$ is a Cohen-Macaulay module over $S(\mfg^*)^G$.
\end{defn} 
To motivate this notion, let us restrict to the action of a torus $T$. (Note as well that the Cohen-Macaulay property for the action of a compact, connected Lie group $G$ is equivalent to that of the restriction of the action to a maximal torus, see \cite[Proposition 2.9]{GR}.) It turns out that the Cohen-Macaulay property is equivalent to the exactness of an Atiyah-Bredon-type sequence
\[
0 \to H^*_T(M) \to H^*_T(M_b) \to H^{*+1}_T(M_{b+1},M_b) \to \cdots \to H^{*+k}_T(M_k,M_{k-1})\to 0,
\]
where $b$ is the lowest occurring orbit dimension, see \cite{GT} or \cite[Section 5]{FranzPuppe2}. In particular, the equivariant cohomology algebra, for Cohen-Macaulay actions, is computable as for equivariantly formal actions, by determining the image of the restriction map $H^*_T(M)\to H^*_T(M_b)$. Note however that the natural map $H^*_T(M)\to H^*(M)$ is not surjective for Cohen-Macaulay actions, which is why this notion is less useful for computing the ordinary cohomology of a $T$-manifold (however, one may divide both the acting torus and the manifold by a locally freely acting $b$-dimensional subtorus to obtain an equivariantly formal action for which the considerations of Section \ref{sec:HMHTM} hold true).

For torus actions with fixed points, or more generally for $G$-actions with points with maximal isotropy rank the notion of being Cohen-Macaulay coincides with equivariant formality \cite[Proposition 2.5]{GR}.

Many geometrically important classes of actions are Cohen-Macaulay. Besides the already known classes of equivariantly formal actions, like Hamiltonian actions on symplectic manifolds, see Example \ref{ex:morsebotteqformal}, they include:
\begin{enumerate}
\item $G$-actions for which all points have the same isotropy rank \cite[Corollary 4.3]{GR}, in particular, transitive $G$-actions.
\item Actions of cohomogeneity one \cite{GM1}. One can also determine the multiplicative structure of the equivariant cohomology of cohomogeneity one manifolds explicitly, see \cite{CGHM}. Note that cohomogeneity-two actions are not necessarily Cohen-Macaulay; an easy example is a $T^2$-action on $(S^1\times S^3)\#(S^2\times S^2)$ with exactly $2$ fixed point (see \cite{OrlikRaymond} and \cite[Example 4.3]{GM1}). 
\item The action of the closure of the Reeb flow of a $K$-contact manifold \cite{GNT}.
\item Hyperpolar actions on symmetric spaces \cite{GHM}. 
\end{enumerate}

\section{Actions on foliated manifolds}

The main algebraic ingredient of the construction of the Cartan model is the structure of a $G$-differential graded algebra on $\Omega(M)$ induced by a $G$-action on $M$. That is, the $G$-action induces contraction operators $i_X$ and Lie derivative operators $L_X$, for every $X\in \mfg$, on $\Omega(M)$. It was Cartan's original approach to abstract from the concrete geometric setting, and consider equivariant cohomology of abstract $G$-differential graded algebras, see \cite[Section 4]{Cartan1}.

In \cite{GT2} this was applied this to foliated manifolds, using the notion of transverse action from \cite[Section 2]{AlvarezLopez}:
\begin{defn}
A \emph{transverse action} of a finite-dimensional Lie algebra $\mfg$ on a foliated manifold $(M,\cF)$ is a Lie algebra homomorphism
\[
\mfg\longrightarrow l(M,\cF).
\]
\end{defn}
Here, $l(M,\cF)= L(M,\cF)/\Xi(\cF)$ is the Lie algebra of \emph{transverse fields}: $L(M,\cF)$ is the Lie algebra of \emph{foliate fields}, i.e., vector fields whose flows send leaves to leaves, which is the same as the normalizer of the subalgebra of vector fields $\Xi(\cF)$ tangent to $\cF$ in the Lie algebra $\Xi(M)$ of all vector fields on $M$. For the trivial foliation by points, a transverse action is the same as an ordinary infinitesimal action on $M$.

Recall that on a foliated manifold $(M,\F)$  the $\F$-basic forms
\[
\Omega(M,\cF)=\{\omega\in \Omega(M)\mid i_X\omega=\cL_X\omega=0 \text{ for all }X\in \Xi(\F)\}
\] 
define, in the same way as the $G$-basic forms introduced in Definition \ref{defn:basicforms}, a subcomplex of the de Rham complex of $M$, thus yielding the \emph{$\F$-basic cohomology} $H^*(M,\F)$. This cohomology was first considered by Reinhart \cite{Reinhart}.

A transverse action of a finite-dimensional Lie algebra $\mfg$ on a foliated manifold $(M,\cF)$ induces the structure of a $\mfg$-differential graded algebra, thus yielding a notion of \emph{equivariant basic cohomology} \cite{GT2} for transverse actions. Explicitly, one defines on
\[
\Omega_{\mfg}(M,\F) := (S(\mfg^*)\otimes \Omega(M,\F))^{\mfg}
\]
an equivariant differential $d_\mfg$ in the same way as in Definition \ref{def:equivariantdifferential}, and obtains $H^*_\mfg(M,\F)$ as the cohomology of this complex.

The main example for which this variant of equivariant cohomology was investigated was the \emph{Molino action} of a Killing foliation \cite{Molino} (see \cite[Section 4.1]{GT2} for a short summary): this is an action of an abelian Lie algebra $\mfa$ whose orbits are the leaf closures of the foliation. Imitating classical results on the fixed point sets of torus actions as in Section \ref{sec:fixedpoints}, one can use this theory to obtain results about the set of closed leaves of a Killing foliation. For example, one obtains the following generalization of Proposition \ref{prop:aeqformalfixedpointset} \cite{GT2}:

\begin{thm} For any transversely oriented Killing foliation $\cF$ on a compact manifold $M$, the union $C\subset M$ of closed leaves of $M$ satisfies
\[
\dim H^*(C,\cF)\leq \dim H^*(M,\cF),
\]
and equality holds if and only if the Molino action is equivariantly formal.
\end{thm}

On the other hand, there are criteria for equivariant formality of the Molino action, similar to the classical setting. For example we have the following generalization of Example \ref{ex:morsebotteqformal} \cite{GT2}:
\begin{thm}\label{thm:basiceqformal}
If $\cF$ is a transversely oriented Killing foliation on a compact manifold $M$, and $f\colon M\to \RR$ a basic Morse-Bott function whose critical set is the union of closed leaves of $\cF$, then the Molino action is equivariantly formal.
\end{thm}

This criterion was applied to concrete geometric situations such as contact \cite{GNT} or cosymplectic geometry \cite{BaGo} to count closed Reeb orbits. In contact geometry, the existence of a momentum map is automatic, and just as in the symplectic setting, a generic component of the momentum map is a Morse-Bott function. As its critical set is the correct one we can apply Theorem \ref{thm:basiceqformal} to the foliation given by the Reeb vector field (we need $M$ to be $K$-contact in order for the foliation to be Riemannian):
\begin{thm} 
Let $M$ be a compact $K$-contact manifold, and $C\subset M$ the union of closed Reeb orbits. Then
\[
\dim H^*(C,\cF) = \dim H^*(M,\cF).
\]
In particular, if the number of closed Reeb orbits is finite, then it is given by $\dim H^*(M,\cF)$.
\end{thm}
On a compact $K$-contact manifold $(M,\alpha)$ of dimension $2n+1$, the elements $1,[d\alpha],\ldots,[d\alpha]^n$ are nonzero in $H^*(M,\cF)$; in this way we obtain an alternative proof of the statement due to Rukimbira \cite[Corollary 1]{Rukimbira} that the Reeb flow of any compact $K$-contact manifold has at least $n+1$ closed Reeb orbits. Moreover, by an easy application of the Gysin sequence, we find:
\begin{thm}
Let $M$ be a compact $K$-contact manifold of dimension $2n+1$ with only finitely many closed Reeb orbits. Then the number of closed Reeb orbits is $n+1$ if and only if $M$ is a real cohomology sphere,
\end{thm} 

Similar results can be derived in other geometries where there naturally appears a Riemannian foliation, such as $K$-cosymplectic geometry (see \cite[Section 8]{BaGo}).

\appendix
\section{Spectral sequences and the module structure on equivariant cohomology}

We present the basics of the spectral sequence of a filtration and apply them to the Cartan model of equivariant cohomology. By also paying attention to the multiplicative structure on spectral sequences, this tool allows us to derive some fundamental properties of the $S(\mathfrak{g}^*)^G$-module structure on $H^*_G(M)$: it is finitely generated and its rank agrees with that of the final page of the spectral sequence associated to a certain filtration. Also, we use spectral sequences to prove the torus case of Remark \ref{rem:commutingactionprinciple}.
Finally we give an example where $E_\infty$ and $H^*_G(M)$ are not isomorphic as $S(\mathfrak{g}^*)^G$-modules, a point which is in several places unclear in the literature.

Before we start, we want to point out that the goal here is not to give a complete introduction to spectral sequences but rather to provide the reader with all the algebraic background that is needed for our (and many other topological) applications.
In particular, we avoid the finer details of convergence by restricting to first-quadrant spectral sequences. For an in-depth introduction we recommend, e.g., Chapter 5 of \cite{Weibel}.

\subsection{Basic definitions}

Let $R$ be a commutative ring. When applying algebraic results to equivariant cohomology we will always take $R=\mathbb{R}$.

\begin{defn}
A \emph{(cohomology) spectral sequence} is a sequence $\{(E_r,d_r)\}_{r\geq 0}$ of bigraded $R$-modules $E_r=\bigoplus_{p,q\in\mathbb{Z}}E^{p,q}_r$ with $R$-linear differentials $d_r^{p,q}\colon E_r^{p,q}\rightarrow E^{p+r,q-r+1}_r$ satisfying $d_r\circ d_r=0$ and isomorphisms $E_{r+1}^{p,q}\cong (\ker d_r^{p,q})/(\im d_r^{p-r,q+r-1})$.
\end{defn}

A spectral sequence is often compared to a book, where for turning the $r$th page $E_r$ one takes cohomology to arrive at the next page $E_{r+1}\cong H^*(E_r,d_r)$. The advantage of spectral sequences is that they can be used to approximate the cohomology of a cochain complex by breaking down the transition $(C^*,d)\rightsquigarrow H^*(C^*,d)$ into smaller steps. Let us now make this idea precise by defining a suitable notion of convergence.

A \emph{first-quadrant} spectral sequence is a spectral sequence $(E_r,d_r)$ where $E_r^{p,q}=0$ whenever $p<0$ or $q<0$. Note that if we fix a bidegree $(p,q)$ and start turning through the pages, the differentials $d_r^{p,q}$ (resp.\ $d_r^{p-r,q+r-1}$) eventually leave (resp.\ come from outside) the first-quadrant and thus are trivial. This implies that $E_r^{p,q}\cong E_l^{p,q}$ for all $l\geq r$. This stable value is denoted by $E_\infty^{p,q}$ and the the bigraded $R$-module $E_\infty$ is called the final page of the spectral sequence. If for some $r$ we have $d_i=0$ for $i\geq r$, or equivalently $E_r=E_\infty$, we say that the spectral sequence \emph{collapses} at $E_r$. While we will solely be interested in first-quadrant spectral sequences, the definition of $E_\infty$ is not limited to this special case and makes sense whenever the pointwise limit exists.

\begin{defn}
A \textit{filtration} of a (graded) $R$-module $H$ is a sequence of (graded) submodules
\[\ldots\subset F^pH\subset F^{p-1}H\subset\ldots\]
We say that the spectral sequence $(E_r,d_r)$ \textit{converges} to a graded module $H^*$ if there is a filtration of $H^*$ such that in any degree $n$ we have
\[0=F^sH^n\subset\ldots\subset F^pH^n\subset F^{p-1}H^n\subset\ldots\subset F^tH^n=H^n\] for some $s,t\in \mathbb{Z}$ and
$E_\infty^{p,q}\cong F^pH^{p+q}/F^{p+1}H^{p+q}$.
\end{defn}
Note that when working with $\mathbb{R}$-coefficients (or over any field) there is a highly non-canonical isomorphism of vector spaces $H^n=\bigoplus_p F^pH^n/ F^{p+1}H^n=\bigoplus_{p+q=n}E_\infty^{p,q}$. In particular $H^*\cong E_\infty$ as graded vector spaces when we consider $E_\infty^{p,q}$ to be of degree $p+q$.

\subsection{Spectral sequence of a filtration}\label{Appendix:ConstrSec}
As hinted at above, the usefulness of spectral sequences stems from the fact that they can be used to break the process of taking cohomology down into several steps. Consider, e.g., the Cartan model $C_G(M)=(S(\mathfrak{g}^*)\otimes \Omega(M))^G$ with its differential $d_G=1\otimes d+\delta$ where $\delta$ is the component which raises the degree in $S(\mathfrak{g}^*)$ and $d$ is just the differential on $\Omega(M)$. Algebraically speaking, $C_G(M)$ is a huge and complicated object, but its cohomology under the differential $1\otimes d$ is much smaller (see Prop. \ref{prop:E1term} below). Consequently, when analysing $H_G(M)$, it can be helpful to take cohomology with respect to $1\otimes d$ first, and then worry about the rest of $d_G$. This process of singling out the $1\otimes d$ component is achieved via a suitable filtration and the associated spectral sequence.

\begin{defn}
A filtration of a cochain complex $(C,d)$ of $R$-modules is a family
\[\ldots\subset F^pC\subset F^{p-1}C\subset\ldots\]
of subcomplexes of $C$. The filtration is said to be \textit{canonically bounded} if $F^0C=C$ and $F^{n+1}C^n=0$.
\end{defn}

\begin{rem}\label{Appendix:Hfiltration}
A filtration of a complex $(C,d)$ induces a filtration $F^*H^*(C,d)$ of $H^*(C,d)$, where $F^pH^n(C,d)$ is the image of the map $H^n(F^pC,d)\rightarrow H^n(C,d)$.
\end{rem}

\begin{thm}\label{Appendix:specseq}
Let $(C,d)$ be a cochain complex and $F^*C$ a canonically bounded filtration. Then the construction below gives rise to a first-quadrant spectral sequence $(E_r,d_r)$ converging to $H^*(C,d)$. More precisely we have \[E^{p,q}_\infty\cong F^{p}H^{p+q}(C,d)/F^{p+1}H^{p+q}(C,d),\] where $F^pH^n(C,d)$ is defined as above.
\end{thm}

In the construction we, for the moment, forget about the cohomological degree and focus purely on the filtration degree. The second component of the bidegree will be added in the end. We start by setting
\[E_0^{p}=F^pC/F^{p+1}C.\]
This carries a differential induced by $d$ and $E_0=\bigoplus_p E_0^p$ is known as the associated graded chain complex. Its cohomology $E_1$ is a first approximation of the cohomology of $(C,d)$, where cocycles are represented by elements whose filtration degree increases under the differential. Note that there is a subquotient of $E_0$ that is a much better approximation of the cohomology, namely $E_\infty=\bigoplus_p E_\infty^p$ where
\[E_\infty^p=\frac{\ker d\cap F^pC+F^{p+1}C}{\im d\cap F^pC+F^{p+1}C}.\]
To interpolate between the two we introduce the approximate cycles
\[A_r^p=\{x\in F^pC~|~dx\in F^{p+r}C\}\]
whose filtration degree increases by $r$ under the differential. Now set
\[E_r^p= \frac{A^p_r+F^{p+1}C}{d(A^{p-r+1}_{r-1})+F^{p+1}C} \cong\frac{A_r^{p}}{d\left(A_{r-1}^{p-r+1}\right)+A_{r-1}^{p+1}}.\]
The usefulness of these interpolations stems from the fact that $E_{r+1}$ can be computed from $E_r$: by definition $d$ induces a map $d_r\colon E_r^p\rightarrow E_r^{p+r}$ and one can identify $E_{r+1}$ with $H(E_r,d_r)$ (see \cite[Theorems 5.4.1]{Weibel} for details).

The bigrading in the spectral sequence arises from additionally considering the grading on $C$. We want the latter to correspond to the total degree of the bigrading so we set $A^{p,q}_r=A_r^p\cap C^{p+q}$, which naturally induces a bigrading on $E_r$. Explicitly we have \[E_r^{p,q}=\frac{A_r^{p,q}}{d\left(A_{r-1}^{p-r+1,q+r-2}\right)+A_{r-1}^{p+1,q-1}}.\]
Since $d_r$ raises the total degree by one and the filtration degree by $r$, it is of bidegree $(r,-r+1)$.
To construct the isomorphism $E^{p,q}_\infty\cong F^{p}H^{p+q}(C,d)/F^{p+1}H^{p+q}(C,d)$, note that $d$ vanishes on $A^{p,q}_r$ for $r>q+1$ because the filtration is canonically bounded. Thus $E_r^{p,q}$ is represented by cocycles from $\ker(d)\cap F^pC^{p+q}$. The isomorphism is then defined by just mapping those cocycles onto their image in $F^{p}H^{p+q}(C,d)/F^{p+1}H^{p+q}(C,d)$. For further details like well-definedness of the last map we again refer to \cite[Theorems 5.4.1 and 5.5.1]{Weibel}.


\subsection{The spectral sequence of the Cartan model}\label{Appendix:Section:CartanSpeq}
From now on let $G$ be a compact, connected group acting on a manifold $M$.
Recall from the definitions in Section \ref{sec:CartanModel} that the Cartan model $C_G(M)\subset S(\mathfrak{g}^*)\otimes \Omega^*(M)$ inherits a bigrading via \[\left( S(\mathfrak{g}^*)\otimes \Omega^*(M)\right)^{p,q}= S^\frac{p}{2}(\mathfrak{g}^*)\otimes\Omega^q(M),\] whenever $p$ is even and $C_G^{p,q}(M)=0$ when $p$ is odd. In particular, $S(\mathfrak{g}^*)$ is concentrated in even degrees when considered as the subalgebra $C_G^{*,0}$.
We also assign a total degree via $C_G^n(M)=\bigoplus_{p+q=n}C^{p,q}_G(M)$.
The Cartan differential is $d_G=1\otimes d+\delta$ with $d$ just the regular differential in
$\Omega^*(M)$ and $(\delta\omega)(X)=-i_{\overline{X}}(\omega(X))$. Note that $1\otimes d$ and $\delta$ are themselves differentials of bidegree $(0,1)$ and $(2,-1)$.

\begin{rem}
Doing a suitable degree shift one can achieve that the bidegrees of the differentials are $(0,1)$ and $(1,0)$. With this grading $C_G(M)$ becomes a double complex in the classical sense and the spectral sequence we construct below is (up to degree shifts) the spectral sequence associated to this double complex (c.f.\ \cite{GuilleminSternberg}). As the degree shift will not simplify our presentation of the material and the original bigrading is more in line with the topological conventions, we decide to stick to the original one.
\end{rem}

In what follows we will write $C$ instead of $C_G(M)$. The filtration we consider on $C$ is defined by
\[F^pC:= C^{\geq p,*}=\bigoplus_{l\geq p,q\geq 0}C^{l,q}.\]
It is canonically bounded as \[F^pC^n=\bigoplus_{l=p}^nC^{l,n-l}.\]
The differential $d_G$ restricts to the $F^pC$, so this is indeed a filtration by subcomplexes and we have an associated spectral sequence to which we just refer as the spectral sequence of $C$. Let us now explicitly compute the first pages.

We have $E_0^{p,q}=F^pC^{p+q}/F^{p+1}C^{p+q}$, which is canonically isomorphic to $C^{p,q}$ via the projection onto this summand.
The differential $d_0\colon E^{p,q}_0\rightarrow E^{p,q+1}_0$ is just the one induced by $d_G$ on the quotient. The composition with the isomorphisms
\[C^{p,q}\cong F^pC^{p+q}/F^{p+1}C^{p+q}\xrightarrow{d_G} F^pC^{p+q+1}/F^{p+1}C^{p+q+1}\cong C^{p,q+1}\]
is precisely the its bidegree $(0,1)$ component $1\otimes d$. Thus we see that $(E_0,d_0)$ is isomorphic to $(C,1\otimes d)$ as a cochain complex.

\begin{rem}\label{Appendix:rem:multiplicativebla}
The following observation will become relevant when discussing multiplicative aspects in Section \ref{Appendix:Multstrucsec}.
In fact the above isomorphism $(C,1\otimes d)\cong (E_0,d_0)$ is one of commutative differential graded algebras (cdga, see Section \ref{Appendix:Multstrucsec}) with respect to the product \[F^pC/F^{p+1}C\otimes F^qC/F^{q+1}C\rightarrow F^{p+q}C/F^{p+q+1}C\] on $E_0$ which is induced by multiplication in $C$. The cohomology of a cdga is naturally a commutative graded algebra. Morphisms between cdgas, i.e.\ multiplicative maps that respect the grading and commute with the differential, induce multiplicative maps in cohomology. The isomorphism in the following proposition is of this form and hence respects the algebra structure.
\end{rem}

\begin{prop}\label{prop:E1term} If $G$ is a compact, connected  Lie group acting on a compact differentiable manifold, then the $E_1$-term in the spectral sequence associated to the Cartan complex is
\[
E_1 \cong S(\mfg^*)^G\otimes H^{*}(M).
\]
\end{prop}

\begin{proof}
We just need to compute the cohomology of $(E_0,d_0)$. Consider the inclusion of complexes
\begin{equation}\label{eq:a8}
(C,1\otimes d) = ((S(\mfg^*)\otimes \Omega(M))^G,1\otimes d) \longrightarrow (S(\mfg^*)\otimes \Omega(M),1\otimes d).
\end{equation}
With regards to Remark \ref{Appendix:rem:multiplicativebla} note that it is an inclusion of cdgas. We obtain the induced map on cohomology
\[
i\colon H^*(C,1\otimes d) \longrightarrow S(\mfg^*)\otimes H^*(M).
\]
Let us show first that it is injective. Assume that $\omega \in C$ is such that $\omega=(1\otimes d)(\sigma)$ for some $\sigma\in S(\mfg^*)\otimes \Omega(M)$.  As $\omega$ is $G$-invariant and $1\otimes d$ commutes with the diagonal $G$-action on $S(\mfg^*)\otimes \Omega(M)$, we have $(1\otimes d)(g^*\sigma)=\omega$ for all $g\in G$. But then also
\[
(1\otimes d)\left(\int_G g^*\sigma\, dg\right)= \int_G(1\otimes d)g^*\sigma\, dg = \int_G \omega\, dg = \omega.
\]
Because $\int_G g^*\sigma\, dg\in C$, it follows that $[\omega]=0\in H^*(C,1\otimes d)$.

We next claim that the map $i$ takes values in $S(\mfg^*)^G\otimes H^*(M)$, which means that for every $[\omega]$ on the left hand side, the element $i[\omega]$ is $G$-invariant when considered as a polynomial function with values in $H^*(M)$. For $g\in G$ the diffeomorphism $g^{-1}\colon M\to M$ is homotopic to the identity, because $G$ is connected. Then, for any $X\in \mfg$ we have $[\omega(\Ad_gX)] = [(g^{-1})^*\omega(X)] = [\omega(X)]$.

Finally we show that $i:H^*(C,1\otimes d)\to S(\mfg^*)^G\otimes H^*(M)$ is surjective. For this we precompose \eqref{eq:a8} with the inclusion
\[
(S(\mfg^*)^G\otimes \Omega(M)^G,1\otimes d) \longrightarrow (C,1\otimes d).
\]
In cohomology we obtain the composition
\[
S(\mfg^*)^G\otimes H^*(\Omega(M)^G,d)\longrightarrow H^*(C,1\otimes d)\overset{i}\longrightarrow S(\mfg^*)^G\otimes H^*(M)
\]
which, by Theorem \ref{thm:cohomofinvariantforms}, is an isomorphism. Thus $i$ is surjective.
\end{proof}

\begin{rem}
Note that the proof is simpler in case of a torus action: in this case the coadjoint action on $S(\mft^*)$ is trivial, so the isomorphism $E_1 = S(\mft^*)\otimes H^{*}(M)$ follows directly from Theorem \ref{thm:cohomofinvariantforms}.
\end{rem}

\begin{cor}\label{thm:hoddcollapse}
If the cohomology of $M$ is concentrated in even degrees, i.e., $H^n(M)=0$ whenever $n$ is odd, then the spectral sequence of the Cartan model degenerates at the $E_1$-term.
\end{cor}

\begin{proof}
Under the hypothesis we know that $E_1^{p,q}$ vanishes whenever $p$ or $q$ is odd. Thus $d_1$ vanishes for degree reasons. The same argument applies to all subsequent pages.
\end{proof}

\begin{rem}\label{rem:doddvanishes}
The differential $d_r$ on $E_r$ vanishes whenever $r\geq 1$ is odd, because $S(\mathfrak{g}^*)^G$ is concentrated in even degrees. In particular, the spectral sequence collapses at $E_1$ if and only if it collapses at $E_2$.
\end{rem}

\begin{ex}
Consider the diagonal action of $S^1\subset\mathbb{C}$ on the unit sphere $S^{2n+1}\subset \mathbb{C}^{n+1}$. The Weyl-invariant polynomials are just $\mathbb{R}[u]$, where $u$ is the dual of some generator $X$ of the Lie algebra of $S^1$. The $E_1$ term of the spectral sequence is isomorphic to $\mathbb{R}[u]\otimes H^*(S^{2n+1})$, so it consists just of two copies of $\mathbb{R}[u]$, embedded as $E_1^{*,0}$ and $E_1^{*,2n+1}$. A differential can only be nonzero if it maps from the $(2n+1)^{\mathrm{st}}$ row to the $0^{\mathrm{th}}$ row. Consequently we have $d_r=0$ for $1\leq r\leq 2n+1$ and $E_1\cong E_{2n+2}$. By the same reasoning we have $d_r=0$ for $r\geq 2n+3$ and $E_{2n+3}=E_\infty$. All that remains to understand is what the differential $d_{2n+2}$ does on $E_{2n+2}$:

\begin{center}
\begin{tikzpicture}
  \matrix (m) [matrix of math nodes,
    nodes in empty cells,column sep={0.8cm,between origins},
    row sep={1cm,between origins}]{
           2n+1 &[+2mm]  \mathbb{R}  &  0    &  \mathbb{R}   &   \cdots  & & & &\\
          \vdots     &   &  &  & & & & &\\
          0     &  \mathbb{R} & 0 & \mathbb{R} & \cdots &  \mathbb{R} & 0 & \mathbb{R} &\cdots
          \\[-3mm]
     &   0  &  1  &  2  &  &  2n+2 &\\};
     
     \draw[->,thick] (m-1-2) to (m-3-6) node[midway, above=3mm,right=-3mm] {$d_{2n+2}$};
     \draw[->,thick] (m-1-4) to (m-3-8);
    
   \draw[very thick] (m-1-2)++(-0.4,0.4)--++(0,-3.4);
   \draw[very thick] (m-3-1)++(-0.5,-0.4)--++(7.5,0);
    \end{tikzpicture}
    \end{center}

Often spectral sequence arguments can work entirely without knowing the explicit definition of the differentials if one adds an extra ingredient. In this case for example, we know by Theorem \ref{thm:eqcohomlocallyfreeactions} that $E_\infty$ is the cohomology of a $2n$-dimensional manifold and vanishes in degrees above $2n$. This knowledge implies that no elements of greater (total) degree  must survive the transition from $E_{2n+2}$ to $E_{2n+3}$. Consequently $d_{2n+2}\colon \smash{E_{2n+2}^{p,2n+1}}\rightarrow \smash{E^{p+2n+2,0}_{2n+2}}$ has to be an isomorphism for every $p\geq 0$. All that remains on the page $E_{2n+3}=E_\infty$ is therefore $\mathbb{R}[u]/(u^{n+1})$ in the $0$th row. We have shown that $H^*(\mathbb{C}P^n)\cong H_{S^1}(S^{2n+1})\cong\mathbb{R}[u]/(u^{n+1})$ as graded vector spaces. With the help of the discussion of the $\mathbb{R}[u]$-module and algebra structures from the subsequent sections, one can deduce that this isomorphism is actually one of $\mathbb{R}[u]$-algebras. However, this is false in general and only holds because in the example, $E_\infty$ is concentrated in a single row, implying there is only one step in the filtration of $H_{S^1}(S^{2n+1})$.

Finally, let us examine explicitly the generator of $\smash{E_{2n+2}^{0,2n+1}}\cong H^{2n+1}(S^{2n+1})$. Let $\omega_0$ be a $S^1$-invariant volume form on $S^{2n+1}$. Other than suggested by the isomorphism, $\omega_0$ does not represent a generator of $\smash{E_{2n+2}^{0,2n+1}}$ because $d_{S^1}\omega_0=ui_{\overline{X}}\omega_0$ has filtration degree $2$. So $\omega_0$ is not an element of $\smash{A^{0,2n+1}_{2n+2}}$. However, we find a form $\omega_1$ such that $i_{\overline{X}}(\omega_0)=d\omega_1$ because $H^{2n}(S^{2n+1})=0$. Now $d_G(\omega_0+u\omega_1)=u^2i_{\overline{X}}\omega_1$ lies in filtration degree $4$. Inductively we construct a zigzag $\omega=\omega_0+\cdots+u^n\omega_n$ such that $d_G\omega$ is a multiple of $u^{n+1}$. So $\omega$ lies in $\smash{A^{0,2n+1}_{2n+2}}$ and induces an element of $\smash{E_{2n+2}^{0,2n+1}}$. Using the fact that the bidegree-$(0,2n+1)$ component of $\omega$, which is precisely $\omega_0$, does not lie in the the projection $\im d$ of $\im d_G$ to the $(0,2n+1)$ component, we conclude that $\omega$ descends to a generator.
\end{ex}

\subsection{Multiplicative structure}\label{Appendix:Multstrucsec}

\begin{defn}
A \textit{graded} $R$-\textit{algebra} is an $R$-algebra $A=\bigoplus_{k\in \mathbb{Z}} A^k$ (where $A^k$ are $R$-modules) such that the multiplication map respects the grading, i.e., $A^p\cdot A^q\subset A^{p+q}$. It is called \textit{commutative} if $xy=(-1)^{|x||y|}yx$ for homogeneous elements $x,y$ of degrees $|x|,|y|$. If $d\colon A\rightarrow A$ is an $R$-linear map which raises the degree by $1$ and satisfies $d^2=0$ as well as the graded Leibniz rule
\[d(xy)=dx\cdot y+(-1)^{|x|}x\cdot dy,\]
we call $(A,d)$ a \textit{commutative differential graded algebra} (cdga). A filtration $F^*A$ of $A$ (as a graded $R$-module) is called \textit{multiplicative} if $F^pA\cdot F^lA\subset F^{p+l}A$.
\end{defn}

\begin{rem}\label{Appendix:Multfiltration} The cohomology $H^*(A,d)$ of any cdga $(A,d)$ inherits an algebra structure which turns it into a commutative graded algebra.
If $F^*A$ is a multiplicative filtration of $(A,d)$ by subcomplexes, then the induced filtration on $H^*(A,d)$ (see Remark \ref{Appendix:Hfiltration}) is multiplicative with respect to the induced algebra structure. In this case we have well defined product maps
\[\frac{F^pH^n}{F^{p+1}H^n}\otimes \frac{F^lH^m}{F^{l+1}H^m}\longrightarrow \frac{F^{p+l}H^{n+m}}{F^{p+l+1}H^{n+m}},\]
where we write $H^k$ for $H^k(A,d)$.
\end{rem}

\begin{ex}
The differential forms $(\Omega(M),d)$ and the Cartan model $(C_G(M),d_G)$ are cdgas with the total degree which is the sum of both components of the bidegree. The filtration of the Cartan model as defined in the previous section is a multiplicative filtration.
\end{ex}

We have seen that for a suitably filtered complex $(C,d)$ the last page of the associated spectral sequence carries information on $H^*(C,d)$ and the two are even abstractly isomorphic as vector spaces if we use field coefficients. It is natural to ask if in case of a cdga $(A,d)$, $E_\infty$ carries information on the algebra structure on $H^*(A,d)$. While we cannot expect to have $E_\infty\cong H^*(A,d)$ as algebras, the algebra structure does indeed leave its mark on $E_\infty$ in the following manner.

\begin{thm}\label{Appendix:Thm:MultSpecsec}
Let $(A,d)$ be a cdga with a canonically bounded multiplicative filtration $F^*A$. Then the spectral sequence from Theorem \ref{Appendix:specseq} carries a multiplicative structure, i.e., for any $r$ there exist multiplication maps $\mu_r\colon  E^{p,q}_r\otimes E^{s,t}_r\rightarrow E_r^{p+s,q+t}$ with the following properties:\begin{itemize}
\item $(E_r,d_r)$ is a cdga with respect to the total degree of the bigrading.
\item The multiplication $\mu_{r+1}$ is induced by $\mu_r$ under the isomorphism $E_{r+1}\cong H(E_{r},d_{r})$.
\end{itemize}
In particular we get an induced multiplication on $E_\infty$. Under the isomorphism
\[E_\infty^{p,q}= F^pH^{p+q}(A,d)/F^{p+1}H^{p+q}(A,d),\]
this product coincides with the one described in Remark \ref{Appendix:Multfiltration}.
\end{thm}

Details of the proof are given e.g.\ in \cite[Section 2.3]{McCleary}. Let us just quickly demystify the products $\mu_r$ by giving their definition: in the explicit construction of $E_r^{p,q}$ from Section \ref{Appendix:ConstrSec} one easily checks that multiplication in $A$ restricts to $A_r^{p,q}\otimes A_r^{s,t}\rightarrow A_r^{p+s,q+t}$ and that this descends to quotients inducing the map $\mu_r\colon  E_r^{p,q}\otimes E_r^{s,t}\rightarrow E_r^{p+s,q+t}$ from the above theorem. Finally we want to draw the reader's attention to Remark \ref{Appendix:rem:multiplicativebla}, where we argue that
\[E_1\cong S(\mathfrak{g^*})^G\otimes H^*(M)\]
as algebras.

\subsection{On the module structure of the equivariant cohomology}\label{Appendix:Section:ModuleStructure}
One of the interesting features of equivariant cohomology is that it is not only an algebra over $\mathbb{R}$ but over $S(\mathfrak{g}^*)^G$. As we have seen, multiplicative structures carry over to the spectral sequence, so we can use the latter to analyse the $S(\mathfrak{g}^*)^G$-module structure on $H^*_G(M)$.

As the differential $d_G$ of the Cartan model vanishes on $S(\mathfrak{g}^*)^G\otimes 1$, we have $S^p(\mathfrak{g}^*)^G \subset A^{2p,0}_r$ for all $r$. The degreewise projection onto $E^{2p,0}_r$ yields a map
\[S(\mathfrak{g}^*)^G\rightarrow E_r\]
whose image is the zeroth row $E_r^{*,0}$. On the page $E_1\cong S(\mathfrak{g}^*)^G\otimes H^*(M)$ (see Prop.\ \ref{prop:E1term}) it is just the inclusion of $S(\mathfrak{g}^*)^G\otimes \RR$. Note that we also obtain an induced map $S(\mathfrak{g}^*)^G\rightarrow E_\infty$. These maps are easily checked to be morphisms of algebras. Thus, the $E_r$ carry the structure of $S(\mathfrak{g}^*)^G$-modules.

For degree reasons the differentials $d_r$ vanish on $E_r^{*,0}$ for $r\geq 1$ so by the Leibniz rule we have $d_r(fx)=fd_r(x)$ for any $f\in S(\mathfrak{g}^*)^G,x\in E_r$. The module structure on $E_{r+1}$ is just the one that $H(E_r,d_r)$ inherits from the differential graded $S(\mathfrak{g}^*)^G$-module $(E_r,d_r)$.

\begin{lemma}\label{Appendix:Lem:modulegenerators}
Let $x_1,\ldots,x_k\in E_\infty$ be homogeneous elements (with respect to the bigrading) that generate $E_\infty$ as an $S(\mathfrak{g}^*)^G$-module. Choose representatives $y_1,\ldots,y_k\in H^*_G(M)$ via the isomorphisms
\[E_\infty^{p,q}\cong F^pH^{p+q}_G(M)/F^{p+1}H^{p+q}_G(M).\]
Then the $y_i$ generate $H^*_G(M)$ as an $S(\mathfrak{g}^*)^G$-module.
\end{lemma}

\begin{proof}
Let $c_0\in H^l_G(M)$ be any element. It is contained in some $F^pH_G^l(M)$, so we may consider its image $\overline{c_0}\in E_\infty^{p,l-p}$. We find elements $f_1,\ldots,f_k\in S(\mathfrak{g}^*)^G$ such that
\[\overline{c_0}=\sum f_ix_i.\]
Recall that the multiplication in $E_\infty$ respects the bigrading. We may therefore choose the $f_i$ in such a way that they are homogeneous and if $x_i\in E_\infty^{p-m,l-p}$, we have $|f_i|=m$ (in the grading inherited from the Cartan model) or $f_i=0$. This ensures that $\sum_i f_iy_i$ lies in $F^pH^l_G(M)$.
Now by the description of the multiplicative structure on $E_\infty$ from Theorem \ref{Appendix:Thm:MultSpecsec} one verifies that $\sum_i f_iy_i$ projects to $\overline{c_0}$ in $E_\infty^{p,l-p}$. In particular
\[c_1=c_0-\sum f_iy_i\]
projects to $0$ and thus lies in $F^{p+1}H^l_G(M)$. Now we repeat this process for $c_1$ until eventually $c_{l-p+1}\in F^{l+1}H^l_G(M)=0$. We have written $c_0$ as a linear combination of the $y_i$.
\end{proof}

The following proposition applies in particular to compact manifolds. The proof is taken from \cite[Prop.\ 3.10.1]{AlldayPuppe}

\begin{prop}
If $\dim H^*(M)<\infty$, then $H^*_G(M)$ is finitely generated as an $S(\mathfrak{g}^*)^G$-module.
\end{prop}

\begin{proof}
By Lemma \ref{Appendix:Lem:modulegenerators}, it suffices to show that $E_\infty$ is finitely generated. We have seen that $E_1$ is the free module $S(\mathfrak{g}^*)^G\otimes H^*(M)$. The cohomology $H^*(M)$ is finite-dimensional and in particular $E_1$ is finitely generated as an $S(\mathfrak{g}^*)^G$-module. The ring $S(\mathfrak{g}^*)^G$ is is a polynomial ring (see Section \ref{sec:Coadjoint}). In particular it is Noetherian, which implies that submodules and quotients of finitely generated $S(\mathfrak{g}^*)^G$-modules are again finitely generated, see \cite[Prop.\ 6.5]{AtiyahMac}. Thus if $E_r$ is finitely generated, the same is true for $E_{r+1}=H(E_r,d_r)$: the differential respects the module structure so the cohomology is a quotient of the submodule $\ker d_r$. As the spectral sequence collapses after a finite number of pages (at most $\dim M$), we conclude that $E_\infty$ is finitely generated.
\end{proof}

Note that, since $S(\mathfrak{g}^*)^G$ is concentrated in even degrees, the module structure preserves even and odd degree elements. With regard to the resulting decomposition we have the following

\begin{cor}\label{Appendix:Cor:RankGleich}
If $\dim H^*(M)<\infty$, then the ranks of the $S(\mathfrak{g}^*)^G$-modules $E^{\even}_\infty$ (resp.\ $E_\infty^{\odd}$) and $H^{\even}_G(M)$ (resp.\ $H^{\odd}_G(M)$) coincide.
\end{cor}

\begin{proof}
For a finitely generated graded module $M$ over the polynomial ring $S(\mathfrak{g}^*)^G$, the rank is encoded in its Hilbert-Poincar\'e series $H_M(t)=\sum_i \dim( M^i)\, t^i$:\ the latter takes the form $f(t)\prod_{i=1}^r(1-t^{k_i})^{-1}$ for some $f\in \mathbb{Z}[t]$, where $r$ is the number of variables of $S(\mathfrak{g}^*)^G$ and the $k_i$ are their degrees \cite[Thm.\ 11.1]{AtiyahMac}. The rank is then precisely $f(1)$ (check this for a free module first and then deduce it for general $M$ via a free resolution). As we have already seen, $E_\infty$ and $H^*_G(M)$ are isomorphic as graded vector spaces, so the claim follows.
\end{proof}

\begin{rem}
In the corollary above, it is tempting to argue that a basis of a free submodule in $H^*_G(M)$ projects down to the basis of a free submodule of $E_\infty$. However this is false in general.
\end{rem}

\subsection{Naturality and the comparison theorem}
We briefly discuss maps between spectral sequences and the important comparison Theorem. The latter enables us to prove Remark \ref{rem:commutingactionprinciple} in case $G$ and $H$ are tori. Also, a construction made in said proof is needed in the next and final section.
\begin{defn}
A morphism of spectral sequences $(E_r,d_r)\rightarrow (E_r',d_r')$ is a family of morphisms
$f_r\colon E_r\rightarrow E'_r$, defined for large $r$, that preserve the bigrading, commute with the differentials, and have the property that
$f_{r+1}$ is the map induced by $f_r$ in cohomology.
\end{defn}

In particular, if $E_\infty$ is defined, we obtain a map $f_\infty\colon E_\infty\rightarrow E'_\infty$.
Morphisms of spectral sequences associated to filtrations arise naturally via filtration-preserving maps: Suppose $(C,d)$ and $(C',d')$ are canonically bounded filtered cochain complexes and $f\colon C\rightarrow C'$ is a filtration-preserving chain map. Then $f$ maps $A_r^{p,q}$ (see the construction in Section \ref{Appendix:ConstrSec}) to ${A'}_r^{p,q}$ and induces maps $f_r\colon E_r\rightarrow E_r'$ for $r\geq 0$. One checks directly via the definitions that this is a morphism of spectral sequences. For proofs of this and the theorem below we refer to \cite[Thm.\ 5.5.11]{Weibel}.

\begin{thm}[comparison theorem]\label{Appendix:thm:Comparison}
If, in the above setting, one of the $f_r$ is an isomorphism, then so are all subsequent $f_r$ and $f$ induces an isomorphism in cohomology.
\end{thm}

To illustrate the usefulness of the above theorem, we prove Remark \ref{rem:commutingactionprinciple} in the case of tori:

\begin{prop}\label{Appendix:Prop:QuotientShenanigans}
Let a torus $T=T'\times T''$ act on $M$ in such a way that the restricted action of the $T''$-factor is free. Then there is a map $C_{T'}(M/T'')\rightarrow C_{T}(M)$ of cdgas inducing an isomorphism in cohomology.
\end{prop}

\begin{proof}
It suffices to prove the proposition in case $T=T'\times S^1$. Then the general case $T^n=T^l\times T^r$ follows by induction.
Consider now an action of $T=T'\times S^1$ on $M$ with the $S^1$ factor acting freely. Via the above product decomposition the Lie algebra of $T$ decomposes as $\mathfrak{t}\oplus\mathfrak{t}_1$. In Corollary \ref{cor:orbitspacedings} it was proved that $\Omega(M/S^1)\cong \Omega_{\mathrm{bas}~S^1}(M)\rightarrow C_{S^1}(M)$ induces an isomorphism on cohomology. Note that if we restrict this map to $\Omega(M/S^1)^{T'}$, it will take values in $S(\mathfrak{t}_1^*)\otimes \Omega^{T}(M)$. We want to argue that in the diagram
\[\xymatrix{
\Omega(M/S^1)^{T'}\ar[r]^(0.45){\psi_1}\ar[d]^{\psi_2}& S(\mathfrak{t}_1^*)\otimes \Omega(M)^{T}\ar[d]^{\psi_3} \\
\Omega(M/S^1)\ar[r]^{\psi_4} & C_{S^1}(M)
}\]
the map $\psi_1$ induces an isomorphism in cohomology. By Theorem \ref{thm:cohomofinvariantforms} and Corollary \ref{cor:orbitspacedings} (applied to the proved $S^1$ case) we know that $\psi_2$ and $\psi_4$ induce isomorphisms. Consequently, if we show that $\psi_3$ induces an isomorphism, the same will hold for $\psi_1$.

Filter both complexes,  $S(\mathfrak{t}_1^*)\otimes \Omega^{T}(M)$ and $C_{S^1}(M)$, by the degree of $S(\mathfrak{t}_1^*)$ as we did for the construction of the spectral sequence for $C_{S^1}(M)$ (see Section \ref{Appendix:Section:CartanSpeq}). As $\psi_3$ is $S(\mathfrak{t}_1^*)$-linear it respects the filtration and induces a morphism of spectral sequences. As argued before, the $0$th pages of the spectral sequences are isomorphic to the respective filtered complexes $ S(\mathfrak{t}_1^*)\otimes \Omega^{T}(M)$ and $C_{S^1}(M)$ and one quickly checks that the map between the $0$th pages is just $\psi_3$. On both $0$th pages, the differential $d_0$ is $1\otimes d$, with $d$ the exterior derivative on $\Omega(M)$. The inclusion $\Omega(M)^T\rightarrow \Omega(M)$ factors through $\Omega(M)^{S^1}\rightarrow\Omega(M)$ and both induce isomorphisms in cohomology by Theorem \ref{thm:cohomofinvariantforms}. Consequently the inclusion $i\colon\Omega(M)^{T}\rightarrow\Omega(M)^{S^1}$ induces an isomorphism as well and we deduce that $\psi_3={\id_{S(\mathfrak{t}_1^*)}}{\otimes}\ { i}$ induces an isomorphism on $E_1=H(E_0,d_0)$. Now by the Comparison Theorem \ref{Appendix:thm:Comparison}, $\psi_3$ induces an isomorphism in cohomology.

The final step is to show that the map
\[\varphi\colon C_{T'}(M/S^1)=S(\mathfrak{t}^*)\otimes \Omega(M/S^1)^{T'}
\longrightarrow S(\mathfrak{t}^*)\otimes \left(S(\mathfrak{t}_1^*)\otimes \Omega(M)^{T}\right)=C_{T}(M)\]
defined as $\id_{S(\mathfrak{t}_l^*)}\otimes \psi_1$ induces an isomorphism in cohomology. To see this one proceeds analogously to before: Filter both complexes by the degree of $S(\mathfrak{t}^*)$. Then the $0$th pages will be isomorphic to $C_{T'}(M/S^1)$ and $C_{T}(M)$ (the bigrading on the latter is not the usual one!) and $\varphi$ induces a morphism of spectral sequences which on $E_0$ is just $\varphi$ itself. The differentials $d_0$ are $1\otimes d$ and $1\otimes d_{S^1}$. In particular $\varphi$ induces an isomorphism on the cohomology $E_1$ because $\psi_1$ does so on the right tensor factor. Another application of Theorem \ref{Appendix:thm:Comparison} yields the result.
\end{proof}

\subsection{A counterexample}
\label{sec:counterexample}
In \cite{Skjelbred} it was shown that under certain topological conditions, e.g.\ for compact manifolds, the equivariant cohomology of a $S^1$-action and the final page of the spectral sequence are isomorphic as $S(\mathfrak{t}^*)$-modules. For  tori of greater dimension this is no longer true.
We construct here a $T^2$-action on a compact manifold such that the final page of the spectral sequence associated to the Cartan model is not isomorphic as a graded $S(\mathfrak{t}^*)$-module to the equivariant cohomology.

Consider the standard action of the diagonal maximal torus $T^3$ of $\SU(4)$ by left multiplication, where we identify $(s,t,u)$ with the diagonal matrix with entries $(stu,\overline{s},\overline{t},\overline{u})$. The maximal diagonal torus of $\SU(2)$ is a circle and together they yield a product action of $T^4=S^1\times T^3$ on $\SU(2)\times \SU(4)$. We pull back this action along the homomorphism $T^3\rightarrow T^4,~(s,t,u)\mapsto(s,s,t,u)$. Now we take the quotient of the first circle factor of $T^3$ and consider the action of the middle and right circle factors to obtain an action of $T^2$ on the space
\[M:=(\SU(2)\times\SU(4))/S^1.\] This action has the desired properties as we will now show. In what follows the Lie algebra of the $r$-torus will be denoted $\mathfrak{t}_r$.

As it is our goal to show that $H_{T^2}^*(M)$ and $E_\infty$ are not isomorphic let us begin by pointing out the structural difference in the two modules.

\begin{claim} In $E_\infty$ there exists a nontrivial degree $2$ element which is torsion with respect to some linear polynomial in $S(\mathfrak{t}_2^*)$. The same does not hold for $H^*_{T^2}(M)$.
\end{claim}

To analyse $H^*_{T^2}(M)$ we will use that it is isomorphic to $H^*_{T^3}(N)$, where $N=\SU(2)\times \SU(4)$ with the aforementioned $T^3$-action.
The isomorphism is induced by the cdga morphism
\[\varphi\colon C_{T^2}(M)=S(\mathfrak{t}_2^*)\otimes \Omega(M)^{T^2}\longrightarrow S(\mathfrak{t}_2^*)\otimes \left(S(\mathfrak{t}_1^*)\otimes\Omega(N)^{T^3}\right)=C_{T^3}(N)\]
which was constructed in the proof of Proposition \ref{Appendix:Prop:QuotientShenanigans}, where we decompose $\mathfrak{t}_3$ as $\mathfrak{t}_2\oplus \mathfrak{t}_1$ in such a way that $\mathfrak{t}_1$ corresponds to the subcircle of $T^3$ such that $M=N/S^1$. In the proof we also argued that $\varphi$ induces an isomorphism between the $E_\infty$-term of the spectral sequence of $C_{T^2}(M)$ and the final page $E_\infty'$ of the spectral sequence obtained by filtering $C_{T^3}(N)$ by the degree of $S(\mathfrak{t}^*_2)$. This allows us to work with the latter spectral sequence when analysing the $E_\infty$-term.
Note that under the isomorphisms $H_{T^2}(M)\cong H_{T^3}(N)$ and $E_\infty\cong E_\infty'$, the $S(\mathfrak{t}_2^*)$-module structure on the left side corresponds to the pullback of the $S(\mathfrak{t}_3^*)$-module structure on the right side along the inclusion $S(\mathfrak{t}_2^*)\rightarrow S(\mathfrak{t}^*_3)$.

Now let $X,Y,Z\in \mathfrak{t}_3^*$ be the dual basis of the standard basis of $\mathfrak{t}_3$, with $X$ in the $\mathfrak{t}_1^*$ summand of the decomposition $\mathfrak{t}_3^*=\mathfrak{t}_2^*\oplus \mathfrak{t}_1^*$.
\begin{lemma}
The map $S(\mathfrak{t}^*_3)\rightarrow H_{T^3}^*(N)$ is injective in degrees up to $3$ and its kernel in degree $4$ is generated by $X^2$ and $X^2+XY+Y^2+YZ+Z^2 +ZX$.
\end{lemma}

\begin{proof} Let $(E_r,d_r)$ denote the spectral sequence of $C_{T^3}(N)$. The map $S^p(\mathfrak{t}_3^*)\rightarrow H_{T^3}^{2p}(N)$ factors as
\[S^p(\mathfrak{t}_3^*)\rightarrow E_\infty^{2p,0}\cong F^{2p}H^{2p}_{T^3}(N)\subset H^{2p}_{T^3}(N),\]
where we have used that $F^{2p+1}H^{2p}_{T^3}(N)=0$ (see the definition of the isomorphism at the end of Section \ref{Appendix:ConstrSec}). In particular the kernels of $S(\mathfrak{t}_3^*)\rightarrow E_\infty$ and $S(\mathfrak{t}_3^*)\rightarrow H_{T^3}^*(N)$ coincide.

We have $E_1=S(\mathfrak{t}^*_3)\otimes H^*(\SU(2)\times \SU(4))$. By the K\"unneth formula, $H^*(\SU(2)\times \SU(4))$ is trivial in degrees $1$ and $2$. For degree reasons, no elements in $E_1^{2,0}$ can be hit by a differential, and thus they live to infinity. This shows injectivity. Elements in $E_1^{4,0}$ live to $E_3^{4,0}$, where they can potentially be hit by $d_3\colon E_3^{0,3}\rightarrow E_3^{4,0}\cong H^3(\SU(2)\times\SU(4))$. This is the only nonzero differential entering this position and thus the kernel in degree $4$ corresponds to the image of $d_3$ in $E_3^{4,0}$. In particular it is at most $2$-dimensional because $\dim H^3(\SU(2)\times\SU(4))=2$. It remains to show that the polynomials from the lemma actually lie in the kernel in which case they will span it.

Recall that the $T^3$-action is defined as a pullback of the product $T^4$-action on $N$ along a homomorphism which on Lie algebras is given by $i\colon \mathfrak{t}_3\rightarrow\mathfrak{t}_4,~(x,y,z)\mapsto (x,x,y,z)$ where we use the standard bases. We have a commutative diagram
\[\xymatrix{S(\mathfrak{t}_4^*)\ar[r]\ar[d]  & H_{T^4}^*(N)\ar[d]\\
S(\mathfrak{t}_3^*)\ar[r] & H_{T^3}^*(N)
}\]
induced by the pullback map $i^*\colon C_{T^4}(N)\rightarrow C_{T^3}(N)$.
Let $W,X,Y,Z$ denote the dual basis of the standard basis of $\mathfrak{t}_4$, where $W$ corresponds to the circle factor acting on $\SU(2)$ and $X,Y,Z$ correspond to the maximal torus of $\SU(4)$. Note that $N$ is actually a Lie group and that the $T^4$-action is the action of a maximal torus of $N$. By Remark \ref{rem:modulestructureG/H}, the kernel of $S(\mathfrak{t}_4^*)\rightarrow H_{T^4}^*(N)$ consists of the Weyl-invariant polynomials which in (cohomological) degree 4 are $p_1=W^2$ and $p_2=X^2+XY+Y^2+YZ+Z^2+ZX$. Hence the elements $i^*(p_1),i^*(p_2)$ lie in the kernel of $S(\mathfrak{t}_3^*)\rightarrow H_{T^3}^*(N)$. They are precisely the polynomials from the lemma because $i^*$ maps $W$ to $X$ and $X,Y,Z$ to themselves.
\end{proof}

As we see from the spectral sequence of $C_{T^3}(N)$, the elements $X,Y,Z$ induce a basis of $H^2_{T^3}(N)$. No element of the degree-$4$ part of $\ker(S(\mathfrak{t}^*_3)\rightarrow H_{T^3}^*(N))$ is divisible by a linear polynomial from $S(\mathfrak{t}^*_2)$. Indeed, for an element of the form $aY+bZ$ to divide a nonzero element of the form $cX^2+d(X^2+X(Y+Z)+Y^2+YZ+Z^2)$ it is certainly necessary that $c=-d$ and $a=b$. But $Y+Z$ does not divide $X(Y+Z)+Y^2+YZ+Z^2$.
This proves the claim that no nonzero element of $H_{T^2}^2(M)$ is sent to $0$ by multiplication with a linear polynomial from $S(\mathfrak{t}^*_2)$.

On the contrary, consider the element $\overline{X}\in {E'}_\infty^{0,2}$ induced by $X$ in the spectral sequence obtained by filtering $C_{T^3}(N)$ by the degree in $\mathbb{R}[Y,Z]$ (recall that $E_\infty'$ is isomorphic to the final page associated to $C_{T^2}(M)$). By the lemma, $X(Y+Z)+Y^2+YZ+Z^2$ is a coboundary. But this shows that $X(Y+Z)$ is a coboundary up to elements of filtration degree $4$ and therefore becomes trivial in ${E'}_\infty^{2,2}$. Thus $\overline{X}(Y+Z)=0$. We have shown that $E_\infty$ and $H^*_{T^2}(M)$ are not isomorphic as graded modules.

\end{document}